\documentclass[a4paper,10pt]{article}

\usepackage{amsfonts,amssymb,amsmath}

\usepackage{dsfont}
\usepackage{cancel}
\usepackage[normalem]{ulem}
\usepackage{tikz}
\usepackage{tikz-3dplot,calc}
\usepackage{tikz-cd}
\usepackage{cancel}
\usepackage{lscape}
\usepackage{comment}
\numberwithin{equation}{section}
\input{xy}
\usepackage[all]{xy}
\xyoption{all}
\xyoption{poly}

\newtheorem{theo}{Theorem}[section]
\newtheorem{coro}[theo]{Corollary}
\newtheorem{lemm}[theo]{Lemma}
\newtheorem{prop}[theo]{Proposition}
\newtheorem{defi}[theo]{Definition}
\newtheorem{rema}[theo]{Remark}
\newtheorem{exam}[theo]{Example}

\newenvironment{proof}{\noindent \textbf{{Proof.}} \sf}

\def\qed{\hfill $\diamond$ \bigskip}

\def\C{{\mathcal C}}

\def\H{{\mathcal H}}

\def\M{{\mathcal M}}

\def\s{{\mathcal S}}

\def\lim{\mathop{\rm lim}\nolimits}

\def\Hom{\mathop{\sf Hom}\nolimits}
\def\rad{\mathop{\rm rad}\nolimits}

\def\End{\mathop{\sf End}\nolimits}

\def\Tor{\mathop{\rm Tor}\nolimits}

\def\Ker{\mathop{\rm Ker}\nolimits}

\def\Coker{\mathop{\rm Coker}\nolimits}

\begin{document}
\sf

\title{Han's conjecture and Hochschild homology for null-square projective algebras}
\author{Claude Cibils, Mar\'{\i}a Julia Redondo and Andrea Solotar
\thanks{\footnotesize This work has been supported by the projects  UBACYT 20020130100533BA, PIP-CONICET 11220150100483CO, and MATHAMSUD-REPHOMOL.
The second and third authors are research members of CONICET (Argentina).}}

\date{}
\maketitle
\begin{abstract}
Let  $\H$ be the class of algebras verifying Han's conjecture. In this paper we analyse two types of algebras with the aim of providing an inductive step towards the proof of this conjecture. Firstly we show that if an algebra $\Lambda$ is triangular with respect to a system of non necessarily primitive idempotents, and if the algebras at the idempotents belong to $\H$, then $\Lambda$ is in $\H$. Secondly we consider a $2\times 2$ matrix algebra, with two algebras on the diagonal, two projective bimodules in the corners, and zero corner products. They are not triangular with respect to the system of the two diagonal idempotents. However, the analogous result holds, namely if both algebras on the diagonal belong to $\H$, then the algebra itself is in $\H$.
\end{abstract}



\normalsize

\section{\sf Introduction}

In this paper, a \emph{smooth} finite dimensional algebra over a field $k$ is a finite dimensional algebra of finite global dimension. The word "smooth" is originated in commutative algebra and is convenient for brevity. Observe that in \cite{CUNTZQUILLEN}, for finite dimensional algebras, "smooth" corresponds to algebras of global dimension at most one, that is, hereditary or semisimple algebras.

In 2006, Y. Han conjectured in \cite{HAN} that a finite dimensional algebra whose Hochschild homology vanishes in large enough degrees is smooth. In the same paper Y. Han proved the conjecture for monomial algebras, while in \cite{BERGHMADSEN2009} P.A. Bergh and D. Madsen proved it in characteristic zero for graded finite dimensional local algebras, Kozsul
algebras and graded cellular algebras.  Recently, the same authors showed in \cite{BERGHMADSEN2017} that trivial extensions of selfinjective algebras, local algebras and graded algebras have infinite Hochschild homology, a result which confirms Han's conjecture for these algebras. Observe that in the 90's, the work of the Buenos Aires Cyclic Homology Group \cite{BACH}, and of L. Avramov and M. Vigu\'{e}-Poirrier \cite{AVRAMOVVIGUE} provided the result for finitely generated commutative algebras.

 In relation with Han's conjecture, lower bounds are obtained in \cite{BERGHMADSEN2010} for the dimension of the Hochschild homology groups of fiber products of algebras, trivial extensions, path algebras of quivers containing loops and quantum complete intersections. Note that P.A. Bergh and K. Erdmann  proved in \cite{BERGHERDMANN}  that quantum complete intersections - not at a root of unity - satisfy Han's conjecture. In \cite{SOLOTARVIGUE} A. Solotar and M. Vigu\'{e}-Poirrier proved Han's conjecture for a generalization of quantum complete intersections and for a family of algebras which are in a sense opposite to these. Moreover in \cite{SOLOTARSUAREZVIVAS}, A. Solotar, M. Su\'{a}rez-Alvarez and Q. Vivas considered quantum generalized Weyl algebras  and proved  Han's conjecture for these algebras (out of a few exceptional cases).

In this paper we consider null-square algebras over a field $k$, that is algebras $\Lambda$ of the form $$\left(
  \begin{array}{cc}
    A & N \\
    M & B \\
  \end{array}
\right)$$
where $A$ and $B$ are $k$-algebras, $M$ and $N$ are bimodules, and the product is given by matrix multiplication subject to $MN=0=NM$. For these algebras, $I=M\oplus N$ is a two-sided ideal verifying $I^2=0$ and $C=A\times B$ is a subalgebra. Actually $\Lambda= C\oplus I$,  that is, $\Lambda$ is a cleft singular extension (see  \cite[p. 284]{MACLANE}).

Hochschild homology is a functor $HH_*$ from $k$-algebras to graded vector spaces. Hence for a null-square algebra, $HH_*(C)$ is a direct summand of $HH_*(\Lambda)$. Moreover, note that  $HH_*(C)=HH_*(A)\oplus HH_*(B)$.

In relation to Han's conjecture, this paper treats two opposite cases, one corresponds to quivers without cycles, while in the other case the quiver contains cycles. Both of them aim to provide an inductive step towards proving the conjecture.  In Section \ref{cornerandtriangular} we consider algebras which are $E$-triangular, that is, they do not have oriented cycles with respect to a complete system $E$ of non necessarily primitive orthogonal idempotents - for brevity we call such a set $E$ a "system". In Sections \ref{HHnullsquareprojective} and \ref{cuatro}, on the contrary, we study a case where there is an oriented cycle. In this last case our analysis requires the involved bimodules to be projective.

A null-square algebra with $N=0$ will be called a corner algebra. For these algebras $HH_*(\Lambda)=HH_*(C)$ by a direct computation  that we briefly recall in Section \ref{cornerandtriangular}, see also \cite{LODAY1998} or \cite{CIBILS2000}.  Moreover we show that if a corner algebra is finite dimensional, with $A$ and $B$ smooth, then the corner algebra is also smooth. This leads to our first result, namely corner algebras built on the class of algebras $\H$ verifying Han's conjecture, also belong to $\H$. Note that no extra assumption on $M$ is required in the foregoing.

Based on the previous results,  we go further.  To a system $E$ of  a $k$-algebra $\Lambda$, we associate its Peirce $E$-quiver: the set of vertices is $E$, and for $x\neq y$ elements of $E$, there is an arrow from $x$ to $y$ if $y\Lambda x\neq 0$. If the Peirce $E$-quiver has no oriented cycles then $\Lambda$ is called $E$-triangular. For instance the Peirce $E$-quiver of a corner algebra with respect to the system $E$ given by the two diagonal idempotents is an arrow if $M\neq 0$. We show that if $\Lambda$ is $E$-triangular, then there is a decomposition $HH_*(\Lambda)=\bigoplus_{x\in E} HH_*(x\Lambda x)$. Moreover for a finite dimensional $E$-triangular algebra $\Lambda$  such that $x\Lambda x$ is smooth for all $x\in E$, the algebra is also smooth. We infer that finite dimensional $E$-triangular algebras built on the class $\H$ also belong to $\H$, without requiring additional assumptions on the bimodules $y\Lambda x$.

In Section \ref{HHnullsquareprojective}, we consider null-square algebras $\Lambda$ with non zero  bimodules $M$ and $N$, in other words the Peirce $E$-quiver with respect to the two diagonal idempotents is $\cdot\rightleftarrows\cdot$. If $M$ and $N$ are projective bimodules, $\Lambda$  is called a null-square projective algebra. We provide a long exact sequence computing $HH_*(\Lambda)$, which is associated to the short exact sequence obtained from the product map $\Lambda\otimes_C\Lambda \to \Lambda$. We obtain a projective resolution of the kernel $K^1_C(\Lambda)$ of this map, which enables us to compute $\Tor^{\Lambda\!-\! \Lambda}_{*}(K_C^1(\Lambda),\ \Lambda)$ through invariants or coinvariants of a natural action of cyclic groups $C_m$ on the zero degree Hochschild homology of tensor powers  $N\otimes_BM$, that is $H_0\left( A, \left(N\otimes_BM\right)^{\otimes_{_A}m}\right)^{C_{m}}$ and $H_0\left( A, \left(N\otimes_BM\right)^{\otimes_{_A}m}\right)_{C_{m}}$. We thus obtain the long exact sequence of Theorem \ref{longexactsequence}.

In Section \ref{Han nullsquareprojective} we focus on basic finite dimensional algebras $A$ and $B$ over a perfect field. After choosing a complete system of primitive orthogonal idempotents for each algebra,  the projective bimodules $M$ and $N$ are given explicitly as direct sums of indecomposable projective bimodules. We first prove that if the invariants are zero,  that is $H_0\left( A, \left(N\otimes_BM\right)^{\otimes_{_A}m}\right)^{C_{m}}=0$, then the space itself is zero. We infer that if the null-square projective algebra has zero homology in large degrees,  then the $0$-homology of any tensor power of $N\otimes_AM$ vanishes. Hence the long exact sequence obtained before provides $HH_*(\Lambda)=HH_*(A\times B)$. We prove that the tensor powers of $N\otimes_AM$ and of $M\otimes_BN$ vanish in large enough degrees.  Observe that $H_0\left(A, \left(N\otimes_BM\right)^{\otimes_A*}\right)$ is related to 2-truncated cycles, namely cycles in the Gabriel quiver of a basic algebra in which the product of any two consecutive arrows is zero, as considered in \cite{BERGHHANMADSEN2012} in order to guarantee that Hochschild homology is infinite dimensional.

Another important result that we obtain in this section is the following Theorem \ref{smooth}. For a perfect field $k$, let $\Lambda$ be a finite dimensional null-square projective $k$-algebra, where $A$ and $B$ are smooth. Assuming the bimodules verify $\left(N\otimes_BM\right)^{\otimes_A*}=0$ for large enough exponents, the algebra $\Lambda$ is also smooth. The proof relies on the construction of an explicit projective resolution obtained through successive cones of the identity.

One of the main results of this paper follows: a finite dimensional null-square projective algebra built on the class of basic algebras in $\H$ also belongs to $\H$.

In the last section we give a presentation by quiver and relations of a null-square projective algebra, starting from the same type of presentations of $A$ and $B$. This is useful for producing examples where our results apply.

\section{\sf Han's conjecture for corner and $E$-triangular algebras}\label{cornerandtriangular}

In this section we first consider null-square algebras and their category of representations. Next, we will study corner algebras which are particular cases of null-square algebras, in relation with Han's conjecture. The results that we obtain in this section for corner (and then for triangular  algebras) do not require a projectivity hypothesis on the bimodules considered in the definition of a null-square algebra below.

\begin{defi}
Let $k$ be a field and let $A$ and $B$ be $k$-algebras. Let $M$ and $N$ be respectively a $B-A$-bimodule and an $A-B$-bimodule. The corresponding \emph{null-square algebra} is
$$\left(
  \begin{array}{cc}
    A & N \\
    M & B \\
  \end{array}
\right)$$
where the product is given by matrix multiplication using the products of $A$ and $B$, the bimodule structures of $M$ and $N$, and setting  $mn=0$ and $nm=0$ for all $m\in M$ and $n\in N$.
\end{defi}

\begin{rema}
A \emph{square algebra} is an algebra $$\Lambda = \left(
  \begin{array}{cc}
    A & N \\
    M & B \\
  \end{array}
\right)$$ as before, with two bimodule maps $\alpha:N\otimes_BM\to A$ and $\beta:M\otimes_AN\to B$ verifying the obvious "associativity" conditions that ensure the associativity of the corresponding matrix product on $\Lambda$. A null-square algebra is  a square algebra where $\alpha=0=\beta$. Observe that in \cite{BUCHW}, R.-O. Buchweitz studies  square algebras which are called "(generalised) Morita context" or  "pre-equivalence", and focus on the case where $\alpha$ or $\beta$ are surjective.
\end{rema}

\begin{exam}
Let $\Lambda$ be a $k$-algebra with a decomposition $\Lambda=P\oplus Q$ as  a right $\Lambda$-module. Then $\Lambda$ is a square algebra of the form
$$
\left(
  \begin{array}{cc}
    \End_{\Lambda}P & \Hom_{\Lambda}(Q,P) \\
\Hom_\Lambda(P,Q) & \End_{\Lambda}Q \\
  \end{array}
\right)$$
If for all $f\in\Hom_{\Lambda}(P,Q)$ and for all $g\in\Hom_{\Lambda}(Q,P)$ the compositions $gf$ and $fg$ are zero, the algebra is null-square. \end{exam}

\begin{rema}
Any square algebra $\Lambda$ is obtained as above by considering the right module decomposition:
$$\left(
  \begin{array}{cc}
    A & N \\
    M & B \\
  \end{array}
\right)=\left(
  \begin{array}{cc}
    A & N \\
    0 & 0\\
  \end{array}
\right) \oplus
\left(
  \begin{array}{cc}
    0 & 0 \\
    M & B \\
  \end{array}
\right).$$
\end{rema}

\vskip3mm
Recall that a \emph{cleft singular extension algebra} (see \cite[p. 284]{MACLANE}) is an algebra $\Lambda$ with a decomposition $\Lambda = C \oplus I,$ where $C$ is a subalgebra and $I$ is a two-sided ideal of $\Lambda$ verifying  $I^2=0.$ A null-square algebra $ \Lambda = \left(
  \begin{array}{cc}
    A & N \\
    M & B \\
  \end{array}
\right)$ is an instance of a cleft singular extension with $C=A\times B$ and $I=M\oplus N$. Indeed, $I$ is a two-sided ideal precisely because $MN=NM=0$.

We will next consider systems of idempotents of an arbitrary algebra in order to recall the representation theory of a null-square algebra.

\begin{defi}
Let $\Lambda$ be a $k$-algebra. A \emph{system} of $\Lambda$ is a finite set $E$ of non zero orthogonal idempotents which is complete, \emph{i.e. } $\sum_{x\in E}x=1$.  The system is trivial if $E=\{1\}.$
\end{defi}

Observe  that in the above definition we do not require the idempotents to be primitive.
To  a system $E$ of a $k$-algebra $\Lambda$ we associate a
 $k$-category  $\C_{\Lambda, E}$
 as follows: its objects are the elements of $E$ while the vector space ${}_y\!\left({\C_{\Lambda, E}}\right)_x$ of morphisms from $x$ to $y$ is $y\Lambda x$. The composition is provided by
 the product of $\Lambda$. Of course $\Lambda$ is recovered as the direct sum of all the morphisms spaces of $\C_{\Lambda, E}$, endowed with the matrix product.
It is well known and easy to prove that the  $k$-categories of left $\Lambda$-modules, and of $k$-functors from $\C_{\Lambda, E}$ to $k$-vector spaces, are isomorphic.

Let now $\C$ be a small $k$-category, with set of objects  $\C_0$. Notice that a $k$-functor $\M$ from $\C$ to $k$-vector spaces is given by a family of vector spaces $\{{}_x\M\}_{x\in\C_0}$ and a collection of linear maps
$${}_y\C_x\otimes {}_x\M \stackrel{{}_ym_x} {\xrightarrow{\hspace{15mm} }} {}_y\M$$
such  that, for any objects $x$, $y$ and $z$, the following diagram commutes:
\[
\xymatrix@!C{
{}_z\C_y\otimes {}_y\C_x \otimes {}_x\M         \ar@{->}[d]_{c\otimes 1}          \ar[r]^-{1\otimes {}_ym_x}       &        {}_z\C_y\otimes {}_y\M         \ar@{->}[d]^{{}_zm_y} \\
{}_z\C_x\otimes {}_x\M              \ar[r]_-{{}_zm_x}                &    {}_z\M}
\]
 Next we define a $k$-category, which will be isomorphic to the category of left modules over a square algebra.
\begin{defi}\label{categoryS}

Let $\Lambda = \left(
  \begin{array}{cc}
    A & N \\
    M & B \\
  \end{array}
\right)$ be a square algebra. The objects of the linear category  $\s(\Lambda)$ are $X \underset{\nu }   {\overset{\mu}{\rightleftharpoons}} Y$, where $X$ is an $A$-module, $Y$ is a $B$-module, $X\overset{\mu}{\rightharpoonup}Y$  stands   for a map of $B$-modules $\mu:M\otimes_A X\rightarrow Y$  and analogously
 $X\underset{\nu}{\leftharpoondown}Y$  stands  for a map of $A$-modules $\nu:N\otimes_B Y\rightarrow X$
 which verify
\begin{equation}\label{associativity}
\nu(1_N\otimes \mu)=\alpha\otimes 1_X \mbox{ and } \mu(1_M\otimes \nu)=\beta\otimes 1_Y.
\end{equation}
Note that we identify the vector spaces $A\otimes_AX$ and $X$  through the canonical isomorphism, as well as $Y\otimes_BB$ and $Y$.

A morphism in $\s(\Lambda)$ from $X \underset{\nu }   {\overset{\mu}{\rightleftharpoons}} Y$ to $X'\underset{\nu'}   {\overset{\mu'}{\rightleftharpoons}} Y'$ is a couple $(\varphi,\psi)$ where $\varphi :X\to X'$ is a morphism of $A$-modules,  $\psi :Y\to Y'$ is a morphism of $B$-modules such that  the following diagrams commute:
$$
\xymatrix@!C{
M\otimes_A X       \ar@{->}[d]_{1\otimes\varphi}          \ar[r]^-{\mu}       &       Y        \ar@{->}[d]^{\psi} \\
M\otimes_A X'              \ar[r]_-{\mu'}                &    Y'}
\hskip2cm
\xymatrix@!C{
X   \ar@{->}[d]_{\varphi} &   N\otimes_B  Y  \ar@{->}[l]_-{\nu}     \ar@{->}[d]^{1\otimes \psi}\\
X'   &   N\otimes_B  Y'  \ar@{->}[l]^-{\nu'}  }
$$
\end{defi}
\begin{prop}\label{modulesandcategoryS}
Let $\Lambda$ be a square algebra. The category of left $\Lambda$-modules is isomorphic to $\s (\Lambda)$.
\end{prop}
\begin{proof}
Consider the complete set of orthogonal idempotents $E=\{e, 1-e\}$ of $\Lambda$, where $e=\left(
                                                                                           \begin{array}{cc}
                                                                                             1 & 0 \\
                                                                                             0 & 0 \\
                                                                                           \end{array}
                                                                                         \right).$
The result is an immediate consequence of the previous observations.\qed
\end{proof}

In what follows the categories of  the above  proposition will be identified. Note that for a null-square algebra, the equalities (\ref{associativity}) become
\begin{equation}
\nu(1_N\otimes \mu)=0\mbox{ and } \mu(1_M\otimes \nu)=0.
\end{equation}

\begin{lemm}\label{projectiveone}
Let $\Lambda$ be a square algebra  and let $P$ be a projective $A$-module. The $\Lambda$-module  $\left(P \underset{\alpha}   {\overset{1}{\rightleftharpoons}} M\otimes_A P\right)$ is projective.
\end{lemm}
\begin{proof}
 Let $\Lambda_1$ be the $\Lambda\! -\! A$-bimodule given by the first column of $\Lambda$,  that is,  $\Lambda_1=\left(
                                                                                           \begin{array}{cc}
                                                                                             A & 0 \\
                                                                                             M & 0 \\
                                                                                           \end{array}
                                                                                         \right)= \left(A \underset{\alpha}   {\overset{1}{\rightleftharpoons}} M\right).$
 Note that $\Lambda_1=\Lambda e$ is a projective $\Lambda$-module.  Moreover, if  $X$ is an $A$-module, $ \Lambda_1\otimes_A X = \left(X \underset{\alpha\otimes 1_X}   {\overset{1}{\rightleftharpoons}} M\otimes_A X\right)$. Since $ \Lambda_1\otimes_A A$ is isomorphic to  $\Lambda_1$, we infer that $ \Lambda_1\otimes_A P$ is a projective $\Lambda$-module.
\qed
\end{proof}

The analogous result holds for $B$-modules.

From now on we focus on null-square algebras.
\begin{prop}\label{simples}
Let $\Lambda = \left(
  \begin{array}{cc}
    A & N \\
    M & B \\
  \end{array}
\right)$ be a null-square algebra where $A$, $B$, $M$ and $N$ are finite dimensional. A simple $\Lambda$-module is isomorphic to
$$ S \underset{}   {\overset{}{\rightleftharpoons}} 0 \mbox{ or } 0 \underset{}   {\overset{}{\rightleftharpoons}} T$$
where $S$ and $T$ are simple  $A$ and $B$-modules respectively.
\end{prop}
\begin{proof}
We assert that the Jacobson radical of $\Lambda$ is  $\left(
  \begin{array}{cc}
    \rad A & N \\
    M & \rad B \\
  \end{array}
\right)$ where $\rad A$ and  $\rad B$ are the Jacobson radicals of $A$ and $B$. Indeed, this vector space is a nilpotent two-sided ideal, and the quotient of $\Lambda$ by it is semisimple.
\qed
\end{proof}
\normalsize
\begin{defi}
A \emph{corner algebra} $\Lambda$ is a square algebra with $N=0$. In this case, the objects of $\s(\Lambda)$ are denoted by $X {\overset{\mu}{\rightharpoonup}} Y$.
\end{defi}
In  this Section we consider Han's conjecture  for corner algebras first, and secondly for $E$-triangular algebras which will be defined below. We emphasize that for corner algebras we do not make any hypothesis on the projectivity of $M$. First we recall the following result.
\begin{prop} \cite[Proposition 10, p.86]{EILROSZEL}
Let $A$ and $B$ be finite dimensional smooth $k$-algebras. The $k$-algebra $A\otimes B$ is smooth.
\end{prop}
\begin{theo}\label{cornerfgld}
Let $\Lambda= \left(
  \begin{array}{cc}
    A & 0 \\
    M & B \\
  \end{array}
\right)$ be a corner finite dimensional algebra, where $M$ is a $B\!-\!A$-bimodule. If $A$ and $B$ are smooth, then $\Lambda$ is smooth.
\end{theo}
\begin{proof}
It is well known that if a finite dimensional algebra $A$ is smooth the same holds for $A^{\mathsf{op}}$. By the previous proposition, $B\otimes A^{\mathsf{op}}$ is smooth.

Let
$0\to Q_q\to\cdots\to Q_1\to Q_0\to M\to 0$
be a finite resolution of $M$ by projective $B\!-\!A$-bimodules.

Firstly let $S\rightharpoonup 0$ be a simple $\Lambda$-module where $S$  is a simple $A$-module.
Let
$0\to P_p\to\cdots\to P_1\to P_0\to S\to 0$
be a resolution of $S$ by projective $A$-modules. Observe that the following sequence of $\Lambda$-modules obtained by tensoring the previous resolution by $\Lambda_1$
$$  0\to(P_p \underset{}   {\overset{1} {\rightharpoonup}} M\otimes_A P_p)\to\cdots\to                (P_0 \underset{}   {\overset{1} {\rightharpoonup}} M\otimes_A P_0)             \to (S \underset{}   {\overset{}
{\rightharpoonup}} 0)\to 0$$
 is not exact in general unless $M$ is a projective $A$-module. Instead we consider the double complex obtained by tensoring both resolutions over $A$:

\[
 \xymatrix{
&   & \vdots \ar[d] & \vdots \ar[d] & \vdots \ar[d] &  \\
&  \dots \ar[r] & Q_1\otimes_A P_2 \ar[r] \ar[d] & Q_1\otimes_A P_1 \ar[r] \ar[d] & Q_1\otimes_A P_0 \ar[r] \ar[d]  & 0 \\
&  \dots \ar[r] & Q_0\otimes_A P_2 \ar[r] \ar[d] & Q_0\otimes_A P_1 \ar[r] \ar[d] & Q_0\otimes_A P_0 \ar[r] \ar[d] &  0 \\
&  \dots \ar[r] & M\otimes_A P_2 \ar[r] \ar[d] & M\otimes_A P_1 \ar[r] \ar[d] & M\otimes_A P_0 \ar[r] \ar[d] &  0 \\
&   & 0  & 0  & 0  &
}
\]
 The total complex of this double complex is exact, since each column is obtained by tensoring an exact complex by a projective module. Hence we obtain a finite  exact sequence of $\Lambda$-modules:
\[
 \xymatrix{
& \vdots  & & & \vdots & & \\
&  P_2 \ar[d] &  \hspace{-8mm} \rightharpoonup  & \hspace{-6mm} M\otimes_A P_2  \ar[d] \hspace{3mm} \oplus & \hspace{-6mm} Q_0\otimes_A P_1 \ar[d] \ar[dl] \hspace{3mm} \oplus &  \hspace{-6mm} Q_1\otimes_A P_0  \ar[dl]  \\
&  P_1 \ar[d] & \hspace{-8mm} \rightharpoonup  & \hspace{-6mm} M\otimes_A P_1  \ar[d] \hspace{3mm} \oplus & \hspace{-6mm} Q_0\otimes_A P_0 \ar[dl]  &  &  \\
&  P_0 \ar[d] & \hspace{-8mm} \rightharpoonup  & \hspace{-6mm} M\otimes_A P_0 \ar[d] &  &  &  \\
&  S  \ar[d] & \hspace{-8mm} \rightharpoonup  & 0  &  &    &  \\
&  0 &   &   &   &   &
}
\]
We assert that this is a projective resolution of   $S\rightharpoonup 0$. Indeed the $i$-th module is
$$\left(P_i\rightharpoonup M\otimes P_i\right) \oplus \left(0\rightharpoonup Q_0\otimes_AP_{i-1}\right)\oplus\cdots\oplus \left(0\rightharpoonup Q_{i-1}\otimes_A P_0\right).$$
The first summand $\Lambda_1\otimes_A P_i$ is projective by Remark \ref{projectiveone}. For the other summands, we first notice that if $Q$ is a projective $B\!-\!A$-bimodule and $X$ is any $A$-module,   $Q\otimes_A X$ is a projective $B$-module. Moreover, for a corner algebra $\Lambda$, if $W$ is a projective left $B$-module, then the left $\Lambda$-module $0\rightharpoonup W$ is projective.

 Secondly let $T$ be a simple $B$-module and let $0\rightharpoonup T$ the corresponding simple $\Lambda$-module. Let $R_\bullet \to
 T$ be a  finite $B$-projective resolution of $T$, then $(0\rightharpoonup R_\bullet)\to  (0\rightharpoonup T)$ is a finite  resolution of $0\rightharpoonup T$ by projective $\Lambda$-modules. \qed
\end{proof}

Now, we will define $E$-triangular algebras with respect to a chosen system $E$. We define first a quiver inferred  from the Peirce decomposition $\Lambda=\bigoplus_{x,y\in E}y\Lambda x$.

\begin{defi}\label{Equiver}
Let $\Lambda$ be a $k$-algebra and let $E$ be a system of $\Lambda$. The \emph{Peirce $E$-quiver} $Q_E$ has set of vertices  $E$;  for $x$ and $y$  different elements of $E$ there is an arrow from $x$ to $y$  in case $y\Lambda x\neq 0$. Note that $Q_E$ contains no loops.
\end{defi}
\begin{defi}
An algebra $\Lambda$ is $E$-\emph{triangular} with respect to a non trivial system $E$ if $Q_E$ has no oriented cycles.

\end{defi}

\begin{rema}  In case $|E|=2$, the Peirce $E$-quiver of an $E$-triangular algebra is an arrow, and the algebra is a corner algebra.
Observe that a finite dimensional algebra which is $E$-triangular with respect to a system $E$ may have oriented cycles in its Gabriel quiver.

\end{rema}

\begin{lemm}
Let $\Lambda$ be a $k$-algebra which is $E$-triangular. There exists a system $F$ of two idempotents  such that  $\Lambda$ is a corner algebra.
\end{lemm}

\begin{proof}
The Peirce $E$-quiver has no oriented cycles, it is finite and it has at least two vertices. Then there exists a source vertex $e$,  that is, a vertex with no arrows ending at it. The idempotent  $f=\sum_{x\neq e}x$ is not zero. Since $e\Lambda f=0$, the algebra $\Lambda$ is a corner algebra with respect to the system $F=\{e,f \}$.
\qed
\end{proof}
\begin{coro}\label{triangularfgld}
Let $\Lambda$ be a finite dimensional $k$-algebra which is $E$-triangular with respect to a system $E$. If $x\Lambda x$ is smooth for every $x\in E$, then $\Lambda$ is smooth.
\end{coro}
\begin{proof}
We proceed by induction on the number of vertices. Let $e$ be a source vertex of $Q_E$, let $f=\sum_{x\neq e}x=1-e$ and let $F$ be the system $\{e,f\}$.

Let $E'=E\setminus\{e\}$, which is a system of the algebra $f\Lambda f$. The $E'$-quiver of $f\Lambda f$ has no oriented cycles since $y (f\Lambda f )x = y\Lambda x$ for every $x,y\in E'$.  By hypothesis the algebras  $x(f\Lambda f)x=x\Lambda x$ are smooth for every $x\in E'$. By induction $f\Lambda f$ is smooth.
Theorem \ref{cornerfgld} provides the result since $e\Lambda e$ is smooth and $\Lambda$ is a corner algebra with respect to $F$, as in the proof of the previous lemma.
\qed
\end{proof}
By definition the Hochschild homology vector spaces  of a $k$-algebra $\Lambda$ with coefficients in a $\Lambda$-bimodule $Z$ are
$$H_*(\Lambda, Z)= \Tor^{\Lambda\otimes \Lambda^{\mathsf{op}}}_*(\Lambda, Z)$$
where the later is also denoted by $\Tor^{\Lambda-\Lambda}_*(\Lambda, Z).$

Next we recall the computation of the Hochschild homology of a corner algebra, see for instance \cite{LODAY1998, CIBILS2000}. The following well known result will be  required;  we provide a sketch of its  proof for the convenience of the reader.
\begin{lemm}\label{byseparable}
Let $\Lambda$ be a $k$-algebra, let $D$ be a separable subalgebra of $\Lambda$, let $Z$ be a $\Lambda$-bimodule and let $Z_D= Z/  \langle dz-zd \mid z\in Z, d\in D\rangle$. The homology of the  complex
$$\cdots\stackrel{b}{\to}Z\otimes_{D-D}\left(\Lambda\otimes_D\Lambda\otimes_D\Lambda\right)
\stackrel{b}{\to}Z\otimes_{D-D}\left(\Lambda\otimes_D\Lambda\right)\stackrel{b}{\to}
Z\otimes_{D-D}\Lambda\stackrel{b}{\to}Z_D\to 0$$
\normalsize
is $H_*(\Lambda, Z)$,
where $\otimes_{D-D}$ stands for $\otimes_{{D\otimes D^{\mathsf{op}}}}$ and  where the maps $b$  are given by the usual formulas for computing Hochschild homology:
 \begin{align*}
b(x_0\otimes x_1\otimes \dots\otimes x_n) &= x_0x_1\otimes x_2\otimes  \dots x_n \\
&+\sum_{0}^{n-1}(-1)^i x_0\otimes \dots\otimes x_ix_{i+1}\otimes \dots\otimes x_n \\
&+(-1)^n x_nx_0\otimes x_2\otimes \dots \otimes x_{n-1}.
\end{align*}
\end{lemm}
\begin{proof}
Consider the complex with differential $d$ defined by the usual formulas for the canonical resolution of $\Lambda$ over the ground field
$$\dots \stackrel{d}{\to}\Lambda\otimes_D\Lambda\otimes_D \Lambda\stackrel{d}{\to}\Lambda\otimes_D \Lambda\stackrel{d}{\to}\Lambda\to 0.$$
The map $s$ given by $s(x_1\otimes \dots \otimes  x_n)=1\otimes  x_1\otimes  \dots \otimes  x_n$ is well defined and verifies $ds+sd=1$, this proves that the complex is acyclic. Since $D$ is separable,  $D\otimes D^{op}$ is also separable and any $D$-bimodule is projective. Consequently  the acyclic complex above is a projective resolution of $\Lambda$ by projective $\Lambda$-bimodules.
\normalsize
The statement of the lemma is obtained by applying the functor $Z\otimes_{D-D} -$ to this resolution and observing that  $Z\otimes_{\Lambda - \Lambda}\left(\Lambda\otimes_DX\otimes_D\Lambda\right)$ is canonically isomorphic to $Z\otimes_{D-D} X$ for any $D$-bimodule $X$.
\qed
\end{proof}
\begin{theo}\cite{LODAY1998,CIBILS2000}\label{diago} Let $\Lambda = \left(
                                                    \begin{array}{cc}
                                                      A & 0 \\
                                                      M & B \\
                                                    \end{array}
                                                  \right)$ be a corner algebra, where $A$ and $B$ are $k$-algebras and $M$ is a $B\!\!-\!\!A$-bimodule. There is a decomposition
                                                  $$HH_*(\Lambda)=HH_*(A)\oplus HH_*(B)$$
\end{theo}
\begin{proof}
Let $e$ be the idempotent $\left(
                            \begin{array}{cc}
                              1_A & 0 \\
                              0 & 0 \\
                            \end{array}
                          \right)$
and let $f=1-e=\left(
                            \begin{array}{cc}
                              0 & 0 \\
                              0 & 1_B \\
                            \end{array}
                          \right)$. Let $D=\left(
                          \begin{array}{cc}
                            k & 0 \\
                            0 & k \\
                          \end{array}
                        \right) = ke\times kf$; note that $D$ is a separable subalgebra of $\Lambda$. We assert that the  complex  of the previous lemma is actually the direct sum of the complexes that compute $HH_*(A)$ and $HH_*(B)$.
Indeed, notice that the $D$-bimodule decomposition $\Lambda=A\oplus B\oplus M$ provides a direct sum decomposition
$$\Lambda\otimes_{D-D}\left( \Lambda\otimes_D\cdots\otimes_D\Lambda\right)= (A\otimes\cdots\otimes A) \oplus (B\otimes\cdots\otimes B)$$
since
                        $$0=M\otimes_{D-D}B=M\otimes_{D-D}A=M\otimes_{D-D}M=B\otimes_{D-D}M=A\otimes_{D-D}M$$
and $A\otimes_{D-D}A= A\otimes A$ while $B\otimes_{D-D}B= B\otimes B$. Observe that  in degree $0$ we obtain $\Lambda\otimes_{D-D}D= A\oplus B$.
\qed
\end{proof}
\begin{coro}\label{triangularHH}
For any $k$-algebra  $\Lambda$  which is $E$-triangular with respect to a system $E$, there is a decomposition
$$HH_*(\Lambda)=\bigoplus_{x\in E}HH_*(x\Lambda x).$$
\end{coro}
\begin{proof}
The idea of the proof is similar to the proof of Corollary \ref{triangularfgld}. It follows by induction once a source vertex of $Q_E$ is chosen. \qed
\end{proof}
Next we turn to Han's conjecture  that we recall: if $A$ is a finite dimensional algebra over a field such that $HH_n(A)=0$ for  $n$ large enough, then $A$ is smooth.
\begin{theo}
Finite dimensional corner $k$-algebras built on the class of $k$-algebras $\H$ verifying Han's conjecture also belong to $\H$. \end{theo}
\begin{proof}
Let $\Lambda = \left(
                                                    \begin{array}{cc}
                                                      A & 0 \\
                                                      M & B \\
                                                    \end{array}
                                                  \right)$
be a finite dimensional  corner algebra and suppose  $HH_*(\Lambda)=0$  for large enough degrees. Theorem \ref{diago} shows that the same holds for $A$ and $B$. Since $A$ and $B$ belong to $\H$, they are smooth. By Theorem \ref{cornerfgld}, $\Lambda$ is smooth.
\qed
\end{proof}
\begin{coro}\label{Hantriangular}
Let $\Lambda$ be a finite dimensional $k$-algebra which is $E$-triangular with respect to a system $E$ of $\Lambda$. If for every $x\in E$ the algebras $x\Lambda x$ belong to $\H$, then $\Lambda$ belongs to $\H$.
\end{coro}
\begin{proof}
The proof follows from Corollaries \ref{triangularfgld} and \ref{triangularHH}.\qed
\end{proof}
\begin{rema}\label{agree}
Let $\Lambda$ be a smooth finite dimensional algebra such that $\Lambda/\rad \Lambda$ is a product of copies of the ground field $k$ and which admits a Wedderburm decomposition $\Lambda = D \oplus \rad \Lambda$ where $D$ is a subalgebra of $\Lambda$. Note that if $k$ is perfect a Wedderburm decomposition always exists.  If $\Lambda$ is smooth, it is proven by B. Keller in \cite[2.5]{KELLER} that there is a $K$-theoretical equivalence between $\Lambda$ and $D$. In particular the cyclic homologies of these algebras are isomorphic, as well as the Hochschild homologies due to the Connes' long exact sequences relying cyclic and Hochschild homologies, for $\Lambda$ and $\Lambda/\rad\Lambda$, see for instance \cite{WEIBEL}. Consequently the Hochschild homology of $\Lambda$ is concentrated in degree zero. In this situation, it follows from Han's conjecture that if the Hochschild homology vanishes in large enough degrees, then it actually vanishes in all positive degrees.

We observe that in the situation of  Corollary \ref{Hantriangular}, the result that we have proven agrees with the previous observation. Indeed, we have shown using Corollary \ref{triangularHH} that Hochschild homology is the direct sum of the Hochschild homologies at the idempotents of the system.

\end{rema}

\section{\sf Hochschild homology of null-square projective algebras}\label{HHnullsquareprojective}

 In this section we consider a \emph{null-square projective algebra} $\Lambda$,  that is, a null-square algebra  $\Lambda =\left(
                                                                      \begin{array}{cc}
                                                                        A & N \\
                                                                        M & B\\
                                                                      \end{array}
                                                                    \right)$
where $M$ and $N$ are  projective $B\!-\!A$ and $A\!-\!B$-bimodules respectively; we recall that $MN=NM=0$. We will provide a long exact sequence which computes $HH_*(\Lambda)$.

First we consider a cleft extension algebra $\Lambda=C\oplus I$, where $C$ is a subalgebra and $I$ is a two-sided ideal, see \cite[p. 284]{MACLANE}. Let
$$K^1_C(\Lambda)=\Ker \left(\Lambda\otimes_C\Lambda\stackrel{d}{\longrightarrow}\Lambda\right)$$
where $d$ is given by the product of $\Lambda$. In case $I$ is projective as a $C$-bimodule we will provide a resolution of $K^1_C(\Lambda)$ by projective $\Lambda$-bimodules. This resolution specialized to a null-square projective algebra will allow  to compute  $\Tor_*^{\Lambda-\Lambda}(K^1_C(\Lambda),\Lambda)$. The  mentioned long exact sequence {will be obtained as}  the $\Tor$ exact sequence associated to the short exact sequence of $\Lambda$-bimodules
\begin{equation}\label{theshort}
   0\longrightarrow K^1_C(\Lambda)\longrightarrow\Lambda\otimes_C\Lambda\longrightarrow\Lambda\longrightarrow 0.
\end{equation}
\begin{rema}
This short exact sequence splits as a sequence of $C$-bimodules but it does not split as a sequence of $\Lambda$-bimodules.
\end{rema}
\begin{lemm}
Let $\Lambda=C\oplus I$ be a cleft extension algebra. The  following complex is acyclic:
$$\cdots\stackrel{d}\longrightarrow\Lambda\otimes_CI\otimes_CI\otimes_C\Lambda
\stackrel{d}\longrightarrow\Lambda\otimes_CI\otimes_C\Lambda
\stackrel{d}\longrightarrow\Lambda\otimes_C\Lambda
\stackrel{d}{\longrightarrow}\Lambda
\longrightarrow 0$$
with differentials for $n\geq 3$
\begin{align*}
d(l_1\otimes x_2\otimes \dots\otimes x_{n-1}\otimes l_n) &=  l_1x_2\otimes x_3\otimes  \dots x_{n-1}\otimes l_n \\
&+\sum_2^{n-2}(-1)^{i+1} l_1\otimes \dots\otimes x_ix_{i+1}\otimes \dots\otimes l_n \\
&+(-1)^n l_1\otimes x_2\otimes \dots \otimes x_{n-1}l_n
\end{align*}
and, for $n=2$, the product of the algebra is denoted by $d$ as before.

\end{lemm}

\begin{proof}
Let $l\in\Lambda$ and let $l=l_C+l_I$ be its decomposition in $C\oplus I$. Let $s$ be the map given as follows:
$$s( l_1\otimes x_2\otimes \dots\otimes x_{n-1}\otimes l_n) = 1\otimes \left(l_1\right)_I\otimes x_2\otimes \dots\otimes x_{n-1}\otimes l_n.$$
It is straightforward  to check that $s$ is well defined with respect to the tensor products over $C$. The verification that $s$ is a homotopy contraction is not completely trivial, we illustrate this by checking the property in degree two:
\begin{align*}
  ds(l\otimes x\otimes l')= &l_I\otimes x\otimes l' -1\otimes l_Ix\otimes l' + 1\otimes l_I\otimes xl',\\
  sd(l\otimes x \otimes l')= & 1\otimes \left(lx\right)_I\otimes l' - 1\otimes l_I\otimes xl'.
\end{align*}
Note that $(lx)_I = lx= l_Cx+l_Ix$. Hence
\begin{align*}
(ds+sd)(l\otimes x\otimes l')&=l_I\otimes x\otimes l' - 1 \otimes l_Ix\otimes l' + 1\otimes (lx)_I\otimes l'\\
&=l_I\otimes x\otimes l' - 1\otimes l_Ix\otimes l' + 1\otimes l_Cx\otimes l' +1\otimes l_Ix\otimes l'\\
&=l_I\otimes x\otimes l' + 1\otimes l_cx\otimes l'\\
&=l_I\otimes x\otimes l' + l_C\otimes x \otimes l'\\
&= (l_I+l_C)\otimes x \otimes l'\\
&= l\otimes x \otimes l'.
\end{align*}
\qed
\end{proof}
\begin{prop}\label{projresK}
Let $\Lambda=C\oplus I$ be a cleft extension algebra and suppose $I$ is a projective $C$-bimodule. The following is a resolution of $K_C^1(\Lambda)$ by projective $\Lambda$-bimodules:
$$\cdots \stackrel{d}{\longrightarrow}\Lambda\otimes_CI\otimes_CI\otimes_C\Lambda \stackrel{d}{\longrightarrow}\Lambda\otimes_CI\otimes_C\Lambda\stackrel{d}{\longrightarrow}K_C^1(\Lambda)\longrightarrow 0.$$
\end{prop}
\begin{proof}
The complex is acyclic by the previous result. We claim that if $P$ and $Q$ are projective $C$-bimodules, then $P\otimes_C Q$ is also a projective $C$-bimodule. Indeed $(C\otimes C)\otimes_C(C\otimes C)$ is a projective bimodule and the result follows. Consequently $I\otimes_C\cdots\otimes_C I$ is a projective $C$-bimodule. Moreover, if $P$ is a projective $C$-bimodule it is clear that $\Lambda\otimes_C P\otimes_C \Lambda$ is a projective $\Lambda$-bimodule.
\qed

Let $\Lambda$ be a $k$-algebra and $Z$ be a $\Lambda$-bimodule. We recall the following
$$H_0(\Lambda, Z)\ = \ \Lambda \otimes_{\Lambda\otimes \Lambda^{\mathsf{op}}} Z\ = \ \Lambda\otimes_{\Lambda-\Lambda} Z\ = \ Z/\langle\lambda z - z\lambda\rangle $$
where $\langle\lambda z - z\lambda\rangle$ is the vector subspace of $Z$ generated by the set $\{\lambda z - z\lambda\}$ for all $\lambda\in\Lambda$ and $z\in Z$.

\end{proof}
Let $\Lambda$ be an algebra and let $C$ be a subalgebra. Let $U$ be a $C$-bimodule and let $\Lambda\otimes_C U \otimes_C \Lambda$ be the induced $\Lambda$-bimodule. The next result gives a decomposition of the Hochschild homology in degree zero of a cleft algebra  $\Lambda=C\oplus I$  with coefficients in an induced bimodule. We provide a proof for further use.

\begin{prop}\label{h0induced}
Let $\Lambda=C\oplus I$ be a cleft algebra and let $U$ be a $C$-bimodule.
$$H_0(\Lambda,\ \Lambda\otimes_CU\otimes_C\Lambda) = H_0(C,U) \oplus H_0(C,\ I\otimes_CU)$$
\end{prop}
\begin{proof}
The mutual inverse isomorphisms are given by
$$\begin{array}{llll}
a\otimes u\otimes b&\mapsto\ \left(ba\right)_Cu\ &+\ &\left(ba\right)_I\otimes u,\\
u+ x\otimes v &\mapsto\ 1\otimes u \otimes 1\ &+\ &x\otimes v\otimes 1.
 \end{array}
$$
\qed
\end{proof}
We will use next the previous result for $U= I^{\otimes_{_C} n}$. Let
$$I(n) = H_0(C,  \ I^{\otimes_{_C} n}).$$
\begin{coro}
Let $\Lambda=C\oplus I$ be a cleft algebra. There is a  decomposition
$$H_0\left(\Lambda, \ \Lambda \otimes_C I^{\otimes_{_C} n} \otimes_C \Lambda\right) = I(n)\ \oplus \ I(n+1).$$
\end{coro}
\begin{prop}\label{torKcomplex}
Let $\Lambda = C\oplus I $ be a cleft algebra where $I$ is a projective $C$-bimodule. The vector spaces $\Tor_*^{\Lambda\!-\!\Lambda}(K_C^1(\Lambda),\Lambda)$ are the homology spaces of the complex
$$\cdots \stackrel{b}{\longrightarrow} I(n)\oplus I(n+1) \stackrel{b}{\longrightarrow}I(n-1)\oplus I(n)\stackrel{b}{\longrightarrow}\cdots \stackrel{b}{\longrightarrow} I(2)\oplus I(3) \stackrel{b}{\longrightarrow} I(1)\oplus I(2) \longrightarrow 0     $$
where
$$b: I(n)\oplus I(n+1) \to I(n-1)\oplus I(n)$$ is as follows:
\begin{itemize}
\item If  $z_1\otimes\dots\otimes z_n \in I(n)$, then
\begin{align*}
b(z_1\otimes\dots\otimes z_n) &=  z_1\otimes\dots\otimes z_n  \\
&+\sum_{1}^{n-1} (-1)^i z_1\otimes\dots \otimes z_iz_{i+1}\otimes \dots\ \otimes z_n  \\
&+ (-1)^nz_n\otimes z_1\otimes\dots\otimes z_{n-1}
\end{align*}
where the first and the last terms belong to $I(n)$ and the middle sum belongs to $I(n-1)$.
\item If $z_0\otimes\dots\otimes z_n \in I(n+1)$, then
\begin{align*}
b(z_0\otimes\dots\otimes z_n) &= z_0z_1\otimes\dots\otimes z_n  \\
&+\sum_{0}^{n-1} (-1)^i z_0\otimes\dots \otimes z_iz_{i+1}\otimes \dots\ \otimes z_n  \\
&+ (-1)^nz_n z_0\otimes\dots\otimes z_{n-1}
\end{align*}
which belongs to $I(n)$.
\end{itemize}
\end{prop}
\begin{proof}
The formulas are obtained by applying the functor $H_0(\Lambda, -)$ to the projective resolution of $K_C^1(\Lambda)$ of Proposition \ref{projresK}, and by translating the differentials to the present setting through the isomorphisms provided in Proposition \ref{h0induced}.
\qed
\end{proof}
\begin{lemm}\label{Iodd0}
Let $A$ and $B$ be $k$-algebras, let $C=A\times B$ and let $I$ be a $C$-bimodule of the form $I=M\oplus N$ where $M$ is a $B\!-\!A$-bimodule and $N$ is a $A\!-\!B$-bimodule. For $n$ odd, $I(n)=0$.
\end{lemm}
\begin{proof}
First we notice that $M\otimes_CM=0=N\otimes_C N$ since for instance $m \otimes m' = m(1_A,0) \otimes m' = m\otimes (1_A,0)m' = m\otimes 0 =0$.

Moreover $N\otimes_C M = N\otimes_B M$ and $M\otimes_C N = M\otimes_A N$.

Consequently
$$ I^{\otimes_{_C} n} = \left(\cdots \otimes_AN\otimes_BM\otimes_AN\otimes_BM\right)\ \oplus\ (\cdots \otimes_BM\otimes_AN\otimes_BM\otimes_AN)$$
with $n$ tensorands in each summand.
In particular for $n$ odd we have
$$ I^{\otimes_{_C} n} = \left(M\otimes_A\cdots \otimes_BM\otimes_AN\otimes_BM\right)\ \oplus\ (N\otimes_B\cdots \otimes_AN\otimes_BM\otimes_AN) $$
and we assert  that $H_0(C, \ I^{\otimes_{_C} n}) =0$. Indeed  $(1_A,0)x=0$ for every $x\in M$, while $x(1_A,0)=x$, and  $(0,1_B)y=0$, while $y(0,1_B)=y$ for every $y\in N$.
\qed
\end{proof}
\begin{lemm}
In the same situation as in the previous lemma, for $n=2m$,
$$ I^{\otimes_{_C} n} =\left(N\otimes_BM\right)\otimes_A\cdots \otimes_A\left(N\otimes_BM\right)\ \oplus\ (M\otimes_AN)\otimes_B\cdots \otimes_B(M\otimes_AN)$$
$$= \left(N\otimes_BM\right)^{\otimes_{_A}m}\ \oplus \ \left(M\otimes_AN\right)^{\otimes_{_B}m}.$$
\end{lemm}
\begin{coro}\label{even}
Let $A$ and $B$ be $k$-algebras, let $C=A\times B$ and let $I$ be a $C$-bimodule of the form $I=M\oplus N$ where $M$ is a $B\!-\!A$-bimodule and $N$ a $A\!-\!B$-bimodule. The following decomposition holds:
$$I(2m)= H_0\left( A, \left(N\otimes_BM\right)^{\otimes_{_A}m}\right) \ \oplus\ H_0\left( B, \left(M\otimes_AN\right)^{\otimes_{_B}m}\right).$$
\end{coro}
\begin{defi}\label{cyclicaction}
Let $C_m=\langle t\mid t^m=1\rangle$ be a cyclic group of order $m$. The $kC_m$-module structures of $H_0\left( B, \left(M\otimes_AN\right)^{\otimes_{_B}m}\right)$ and $H_0\left( A, \left(N\otimes_BM\right)^{\otimes_{_A}m}\right)$ are given by the following action of $t$ by cyclic permutation:
$$t(x_m\otimes y_m\otimes \cdots x_2\otimes y_2\otimes x_1\otimes y_1) =  x_1\otimes y_1\otimes x_m\otimes y_m\otimes \cdots x_2\otimes y_2,$$
$$t(y_m\otimes x_m\otimes \cdots y_2\otimes x_2\otimes y_1\otimes x_1) =  y_1\otimes x_1\otimes y_m\otimes x_m\otimes \cdots y_2\otimes x_2.$$
\end{defi}
Note that the above actions are not well defined neither on $M\otimes_AN$ nor on  $N\otimes_BM$, on the other hand they are well defined on the $0$-degree homology of these bimodules.

We provide two isomorphisms between  these $kC_m$-modules that will be used in the proof of the next result:
$$
\begin{array}{rlll}
H_0\left( A, \left(N\otimes_BM\right)^{\otimes_{_A}m}\right) &\stackrel{\sigma}{\to} & H_0\left( B, \left(M\otimes_AN\right)^{\otimes_{_B}m}\right)   \\
  y_m\otimes x_m\otimes \cdots \otimes y_1\otimes x_1  & \mapsto & x_1\otimes y_m\otimes x_m\otimes \cdots \otimes y_1,
\end{array}
$$
$$
\begin{array}{rlll}
H_0\left( B, \left(M\otimes_AN\right)^{\otimes_{_B}m}\right) & \stackrel{\tau}{\to}& H_0\left( A, \left(N\otimes_BM\right)^{\otimes_{_A}m}\right)\\
 x_m\otimes y_m\otimes \cdots \otimes x_1\otimes y_1 & \mapsto & y_1\otimes x_m\otimes y_m\otimes\cdots \otimes x_1.
\end{array}
$$
Notice that the compositions $\sigma\tau$ and $\tau\sigma$ are  the  actions of $t$ on the corresponding vector spaces.

Finally we recall that for a group $G$ and a $kG$-module $H$, the invariants (or fixed points) of the action are $H^G=\{x\in H \mid  sx=x \mbox{ for all } s\in G\}.$ The coinvariants are  $H_G=H/\langle sx-x\rangle$ where $\langle sx-x\rangle$ is the vector subspace of $H$ generated by the elements of the form $sx-x$ for all $s\in G$ and $x\in H$. If $G$ is finite and the characteristic of the field does not divide its order, then $H_G$ and  $H^G$ are canonically isomorphic through the action of $\frac{1}{|G|}\sum_{s\in G}s$.
\begin{theo}
Let $\Lambda =  \left(
                                                                      \begin{array}{cc}
                                                                        A & N \\
                                                                        M & B\\
                                                                      \end{array}
                                                                    \right)$
be a null-square projective algebra, and let $I=M\oplus N$. For $m\geq 0$,
$$
\arraycolsep=0,3mm\def\arraystretch{2}
\begin{array}{lll}
\Tor^{\Lambda\!-\! \Lambda}_{2m+1}(K_C^1(\Lambda),\ \Lambda)&=&H_0\left( B, \left(M\otimes_AN\right)^{\otimes_{_B}m+1}\right)^{C_{m+1}} \mbox{\ \  and }
\\
\Tor^{\Lambda\!-\! \Lambda}_{2m}(K_C^1(\Lambda),\ \Lambda)&=
&H_0\left( B, \left(M\otimes_AN\right)^{\otimes_{_B}m+1}\right)_{C_{m+1}}.
\end{array}$$
\end{theo}
\begin{proof}
We recall that for a null-square projective algebra $MN=0=NM$, hence $I^2=0$. Moreover  $I(n)=0$ for $n$ odd, by Lemma \ref{Iodd0}. Consequently the complex of  Proposition \ref{torKcomplex} reduces to
$$\cdots \stackrel{b}{\to} I(6)\stackrel{0}{\to} I(4)\stackrel{b}{\to} I(4)\stackrel{0}{\to} I(2)\stackrel{b}{\to} I(2)\to 0$$
where for $n=2m$
$$b(z_1\otimes\cdots\otimes z_n) =  z_1\otimes\cdots\otimes z_n\ +\  z_{n}\otimes z_1\otimes\cdots\otimes z_{n-1}.$$
Furthermore, the matrix of
$$
\begin{array}{lll}
b:&H_0\left( A, \left(N\otimes_BM\right)^{\otimes_{_A}m}\right) \ \oplus\ H_0\left( B, \left(M\otimes_AN\right)^{\otimes_{_B}m}\right)\\ &\longrightarrow H_0\left( A, \left(N\otimes_BM\right)^{\otimes_{_A}m}\right) \ \oplus\ H_0\left( B, \left(M\otimes_AN\right)^{\otimes_{_B}m}\right)
\end{array}$$
with respect to the decomposition of Proposition \ref{even}  is
$\left(
                                                                      \begin{array}{cc}
                                                                        1 & \tau \\
                                                                        \sigma & 1\\
                                                                      \end{array}
                                                                    \right).$
Moreover,
$$\begin{array}{ll}
\Ker b &=\{(u,v)\mid u+\tau (v) =0 = \sigma (u) + v\}\\
&=\{(u,-\sigma u)\mid u = \tau\sigma u\} \\
&=\{u\mid tu=u\}\\
&=H_0\left( A, \left(N\otimes_BM\right)^{\otimes_{_A}m}\right)^{C_m}.
 \end{array}$$
In order to compute $\Coker b$, note that $(u,v)=-(\tau v, \sigma u)$  holds in $\Coker b$.  Hence $(u,0)=(0,-\sigma u)=(\tau\sigma (u), 0)$. This shows that the map $$H_0\left( A, \left(N\otimes_BM\right)^{\otimes_{_A}m}\right)_{C_m}\to \Coker b$$ given by $u\mapsto (u,0)$ is well defined. Its inverse is given by $(u,v)\mapsto u-\tau(v)$. Hence $\Coker b= H_0\left( A, \left(N\otimes_BM\right)^{\otimes_{_A}m}\right)_{C_m}$.
\qed
\end{proof}
\sf
Towards describing the long exact sequence mentioned above, we consider now some tools of homological algebra to compute $\Tor_*^{\Lambda\!-\!\Lambda}(\Lambda\otimes_C\Lambda, \ \Lambda)$. The next  result will be used for a null-square projective algebra $\Lambda =  \left(
                                                                      \begin{array}{cc}
                                                                        A & N \\
                                                                        M & B\\
                                                                      \end{array}
                                                                    \right)$
and for the inclusion of algebras $C\otimes C^{\mathsf op} \subset\Lambda\otimes \Lambda^{\mathsf op}$, where $C=A\times B$.
\begin{lemm}\label{tor induced}
Let $F\subset D$ be an inclusion of $k$-algebras and suppose $D$ is projective as a left $F$-module. Let $U$ be a right $F$-module and let $U\!\!\uparrow^D=U\otimes_F D$ be the induced right module. Let $Z$ be a left $D$-module and let ${}_F\!\!\downarrow\!\! Z$ be the left $F$-module obtained by restricting the action to $F$. The following holds:
$$\Tor^D_*(U\!\!\uparrow^D,Z)=\Tor^F_*(U, {}_F\!\!\downarrow\!\! Z).$$
\end{lemm}
\begin{proof}
The left hand side functor in the variable $Z$  is characterised by its universal property :
\begin{itemize}
\item $\Tor^D_0(U\!\!\uparrow^D,Z)= U\!\!\uparrow^D\otimes_D Z = U\otimes_F Z$,
\item $\Tor^D_0(U\!\!\uparrow^D,Z)=0$ if $Z$ is projective,
\item A short exact sequence of of $D$-modules provides a long exact sequence.
\end{itemize}
It is clear that the right hand side functor in the variable $Z$ verifies the same properties. Note  that the second property is fulfilled precisely because we assume ${}_F\!\!\downarrow\!\! D$ is projective.
\qed
\end{proof}
\begin{lemm}
Let $\Lambda$ be a null-square projective algebra and let $C=A\times B$. The $C$-bimodule $\Lambda\otimes\Lambda$ is projective.
\end{lemm}
\begin{proof}
Note first that by hypothesis $M$ is a projective $B\!-\!A$-bimodule. It becomes a $C$-bimodule by extending the actions by zero, then $M$ is a projective $C$-bimodule. The same holds for $N$, then $I=M\oplus N$ is a projective $C$-bimodule.

Consider the $C$-bimodule decomposition
$$\Lambda\otimes \Lambda = (C\otimes C) \oplus (C\otimes I) \oplus (I\otimes C)\oplus (I\otimes I).$$

We assert that a projective $C$-bimodule is also projective as a left (or right) $C$-module. Indeed, the free rank-one $C$-bimodule $C\otimes C$ is free as a left (or right) $C$-module. This observation makes the proof of the assertion immediate. We infer that $I$ is projective as a left and as a right $C$-module.

We record  that if $P$ is a projective left $C$-module and $Q$ is a projective right $C$-module, the $C$-bimodule $P\otimes Q$ is a projective $C$-bimodule.

Consequently the four terms of the above direct sum decomposition of the $C$-bimodule $\Lambda\otimes\Lambda$ are projective $C$-bimodules.
\qed
\end{proof}
\begin{theo}
Let $\Lambda=\left(
                                                                      \begin{array}{cc}
                                                                        A & N \\
                                                                        M & B\\
                                                                      \end{array}
                                                                    \right)$
be a null-square projective algebra, and let $C=A\times B$. There is a decomposition
$$\Tor^{\Lambda\!-\!\Lambda}_*(\Lambda\otimes_C\Lambda, \ \Lambda)= HH_*(A) \oplus HH_*(B).$$
\end{theo}
\begin{proof}
We consider the inclusion $C\otimes C^{\mathsf op} \subset\Lambda\otimes \Lambda^{\mathsf op}$. Lemma \ref{tor induced} with $U=C$  provide the following:
$$\begin{array}{ll}
\Tor^{\Lambda\!-\!\Lambda}_*(\Lambda\otimes_C\Lambda, \ \Lambda)&=\Tor^{\Lambda\!-\!\Lambda}_*(\Lambda\otimes_C C\otimes_C\Lambda, \ \Lambda) \\
&= \Tor^{C-C}_*\left(C, \ {}_{C}\!\!\downarrow\!\!\Lambda\!\!\downarrow_C\right)\\
&= H_*(C, \ {}_{C}\!\!\downarrow\!\!\Lambda\!\!\downarrow_C) \\
&= HH_*(C)\oplus H_*(C,M)\oplus H_*(C,N)
\end{array}$$
We assert that $H_*(C,M)=H_*(C,N)=0$. Indeed, let $P_{\bullet} \to A$ be a projective resolution of the $A$-bimodule $A$, and analogously for $Q_\bullet \to B$. Note that $P_\bullet \oplus Q_\bullet \to A\oplus B$ is a projective resolution of the $C$-bimodule $C$, where the $C$-bimodule structure of $P_\bullet$ is obtained by extending the action to $B$ by zero, and analogously for $Q_\bullet$. The functor $M\otimes_{C\!-\!C}-$ applied to $P_\bullet \oplus Q_\bullet$ gives the zero complex by simple arguments already used in the proof of Lemma \ref{Iodd0} and $H_*(C,M)=0$. Analogously $H_*(C,N)=0$. Note that the assertion also follows from  \cite[p. 173]{CARTANEILENBERG}.

In order to prove $HH_*(C)=HH_*(A)\oplus HH_*(B)$, observe that the summands  $A\otimes_{C\!-\!C}Q_\bullet$ and $B\otimes_{C\!-\!C}P_\bullet$ of  $C\otimes_{C\!-\!C}(P_\bullet\oplus Q_\bullet)$  are zero for analogous reasons.
\qed
\end{proof}
The previous results and the exact sequence (\ref{theshort}) provides the following:
\begin{theo}\label{longexactsequence}
Let $\Lambda=\left(
                                                                      \begin{array}{cc}
                                                                        A & N \\
                                                                        M & B\\
                                                                      \end{array}
                                                                    \right)$
be a null-square projective algebra. There is a long exact sequence as follows:
$$
\arraycolsep=0,3mm\def\arraystretch{1,4}
\begin{array}{llrllllllll}
\dots\\
 H_0 \left(A, (N\otimes_BM)^{\otimes\!_{_A} m+1}\right)^{C_{m+1}} &\to &HH_{2m+1}(A)&\oplus &HH_{2m+1}(B)&\to &HH_{2m+1}(\Lambda)&\to\\
H_0 \left(A, (N\otimes_BM)^{\otimes\!_{_A} m+1}\right)_{C_{m+1}} &\to &HH_{2m}(A)&\oplus &HH_{2m}(B)&\to &HH_{2m}(\Lambda)&\to\\
\dots\\
 H_0 \left(A, (N\otimes_BM)^{\otimes\!_{_A} 3}\right)^{C_3} &\to &HH_5(A)&\oplus &HH_5(B)&\to &HH_5(\Lambda)&\to\\
H_0 \left(A, (N\otimes_BM)^{\otimes\!_{_A} 3}\right)_{C_3} &\to &HH_4(A)&\oplus &HH_4(B)&\to &HH_4(\Lambda)&\to\\
H_0 \left(A, (N\otimes_BM)^{\otimes\!_{_A} 2}\right)^{C_2}&\to &HH_3(A)&\oplus &HH_3(B)&\to &HH_3(\Lambda)&\to\\
H_0 \left(A, (N\otimes_BM)^{\otimes\!_{_A} 2}\right)_{C_2} &\to &HH_2(A)&\oplus &HH_2(B)&\to &HH_2(\Lambda)&\to\\
H_0 \left(A, (N\otimes_BM)\right) &\to &HH_1(A)&\oplus &HH_1(B)&\to &HH_1(\Lambda)&\to\\
H_0 \left(A, (N\otimes_BM)\right) &\to &HH_0(A)&\oplus &HH_0(B)&\to &HH_0(\Lambda)&\to 0.\\
\end{array}$$

\end{theo}
\begin{coro}\label{invariantszzero}
Let $\Lambda=\left(
                                                                      \begin{array}{cc}
                                                                        A & N \\
                                                                        M & B\\
                                                                      \end{array}
                                                                    \right)$
be a null-square projective algebra. If $HH_n(\Lambda)=0$ for $n$ large enough, then $$H_0\left( A, \left(N\otimes_BM\right)^{\otimes_{_A}n}\right)_{C_n}=H_0\left( B, \left(M\otimes_AN\right)^{\otimes_{_B}n}\right)^{C_n}=0$$ for $n$ large enough.
\end{coro}
\begin{proof}
Hochschild homology is a functor from the category of algebras to the category of vector spaces. Let $\Lambda=C\oplus I$ where $C$ is a subalgebra of $\Lambda$ and $I$ is a two-sided ideal. In other words, there is an algebra surjection $\Lambda \to C$ which splits in the category of algebras, then  $HH_*(C)$ is a direct summand of $HH_*(\Lambda)$. Consequently, if $HH_n(\Lambda)=0$ for $n$ large enough, then the same holds for $HH_n(C)$. The long exact sequence of the previous theorem provides the result.
\qed
\end{proof}
\begin{rema}
The morphisms induced by the inclusion $K^1_C(\Lambda)\to \Lambda\otimes_C\Lambda$ of the short exact sequence (\ref{theshort}) are zero. Indeed, if $f:M\to M'$ and $g:N\to N'$ are $C$-bimodule morphisms,  we associate functorially a morphism between the corresponding short exact sequences (\ref{theshort}) for the corresponding algebras $\Lambda$ and $\Lambda'$. This induces in term a functorial morphism between the corresponding long exact sequences of  Theorem \ref{longexactsequence}. In particular, for $M'=N'=0$ we infer that the morphisms induced by the inclusion of (\ref{theshort}) factor through zero, hence they are zero.
Consequently there are short exact sequences as follows for $m> 0$:
\small
$$0\to HH_{2m}(A)\oplus HH_{2m}(B)\to HH_{2m}(\Lambda)\to H_0 \left(A, (N\otimes_BM)^{\otimes\!_{_A} m}\right)_{C_{m}}\to 0$$
$$0\to HH_{2m+1}(A)\oplus HH_{2m+1}(B)\to HH_{2m+1}(\Lambda)\to H_0 \left(A, (N\otimes_BM)^{\otimes\!_{_A} m+1}\right)_{C_{m+1}}\to 0.$$
\normalsize

For $m=0$ we obtain that $HH_0(A)\oplus HH_0(B)$  and $HH_0(\Lambda)$ are isomorphic; this can of course be verified by a direct computation.
\end{rema}

\section{\sf Han's conjecture for null-square projective algebras}\label{Han nullsquareprojective}\label{cuatro}

 Our first aim is to prove that if the algebras $A$ and $B$ are finite dimensional and basic, and if the invariants under the action of the cyclic groups $C_m$ on the spaces considered in Theorem \ref{longexactsequence} are zero, then the spaces themselves are zero.

Let $A$ and $B$ be finite dimensional and basic algebras. Let $E$ and $F$ be complete sets of primitive orthogonal idempotents  of $A$ and $B$ respectively. If $k$ is perfect, then
$$\rad \left(B\otimes A^{\mathsf{op}}\right)=B\otimes \rad A^{\mathsf{op}}+\rad B\otimes A^{\mathsf{op}}$$
and $\{g\otimes e\}_{(g,e)\in F\times E}$ is a complete set of primitive orthogonal idempotents of $B\otimes A^\mathsf{op}$. Consequently $$\{Bg\otimes eA\}_{(g,e)\in F\times E}$$ is a complete set of representatives, without repetitions, of the isomorphism classes of projective $B\!-\! A$-bimodules.
Let
\begin{equation}\label{M}
{}_BM_A= \bigoplus_{(g,e)\in F\times E}  {}_g m_e  \left(Bg\otimes eA\right)
\end{equation}
 be a projective finitely generated $B\!-\! A$ - bimodule, where by the Krull-Schmidt Theorem, the integers ${}_g m_e$ are uniquely determined by $M$. Similarly, let
\begin{equation}\label{N}
{}_AN_B= \bigoplus_{(f,h)\in E\times F}  {}_f n_h  \left(Af\otimes hB\right)
\end{equation}
be a finitely generated projective $A\!-\! B$ - bimodule.

The next definition is pictured in Figure \ref{NM quiver} on page \pageref{NM quiver}.
\begin{defi}
In the situation considered above, the \emph{$(N,M)$-quiver} is defined as follows: its vertices are $E\cup F$, where we agree to distribute $E$ in a first horizontal floor and $F$ in a ground floor.

There are two sort of arrows:
\begin{itemize}
\item Horizontal, distributed into:\\
- first floor ones, which provides the Peirce $E$-quiver of $A$ (see Definition \ref{Equiver}), and \\
- ground floor ones, namely the Peirce $F$-quiver of $B$.
\item Vertical, distributed into:\\
- down ones, there are ${}_g m_e$ arrows from $e$ to $g$ in one-to-one correspondence with the direct summands $Bg\otimes eA$ of $M$, and\\
- up ones, defined according to $N$ in the analogous way than for $M$.
\end{itemize}
\end{defi}
We agree to write the sequence of arrows of a path from right to left, as for composition of morphisms. Recall that the \emph{length} of a path is the length of the corresponding sequence, and that a \emph{cycle} is a path which starts and ends at the same  vertex. Next we define some particular kinds of paths in the $(N,M)$-quiver.
\begin{defi}\label{balanced}
   Let $\gamma$ be a path of the $(N,M)$-quiver,
\begin{itemize}
\item  $\gamma$ is \emph{balanced} if it does not contain two consecutive horizontal arrows.  In case $\gamma$ starts and ends at the same floor, its \emph{revolution number}  is half of the number of the vertical arrows of the sequence of  $\gamma$.
\item $\gamma$ is $E$-\emph{balanced} if it is balanced and it starts and ends at the first floor,  that is, at $E$-vertices. The set of $E$-balanced paths with revolution number $m$ is denoted by $\mathsf{P}^E_m$.
\item  $\gamma$ is an $E$-\emph{vertical balanced cycle} if it is an $E$-balanced cycle whose first arrow is down vertical. The set of $E$-vertical balanced cycles with revolution number $m$ is denoted by $\mathsf{CV}^E_m$.
\end{itemize}
\end{defi}

\begin{theo}\label{invariantszero}
Let $A$ and $B$ be basic finite dimensional algebras over a perfect field $k$, and let $M$ and $N$ be projective bimodules as above. Let $C_m$ be the cyclic group of order $m$ with generator $t$ acting by cyclic permutation  on $H_0 \left(A, (N\otimes_BM)^{\otimes_A m}\right)$ as given in Definition \ref{cyclicaction}.

If $H_0 \left(A, (N\otimes_BM)^{\otimes_A m}\right)^{C_m}=0$, then  $H_0 \left(A, (N\otimes_BM)^{\otimes_A m}\right)=0.$
\end{theo}

\begin{proof}
We assert that $H_0 \left(A, (N\otimes_BM)^{\otimes_A m}\right)$ is a direct sum of vector spaces indexed by $\mathsf{CV}^E_m$.  In order to provide an outline of the evidence, let consider the particular case
$$N= (Af\otimes hB) \oplus (Af'\otimes h'B) \  \mbox{ and }\ M=(Bg\otimes eA) \oplus (Bg'\otimes e'A).$$
Notice that the $(M,N)$-quiver has two down arrows and two up arrows. Then

\begin{equation}\label{NM}
\begin{array}{rrcl}
N\otimes_B M= & (Af\otimes hBg \otimes eA)&\oplus &  (Af\otimes hBg' \otimes e'A)\ \ \oplus\\
                                & (Af'\otimes h'Bg \otimes eA) &\oplus & (Af'\otimes h'Bg' \otimes e'A)
\end{array}
\end{equation}

and
\begin{equation}\label{T1NM}
\begin{array}{lrlll}
H_0 \left(A, (N\otimes_BM)\right)=&(eAf\otimes hBg)&\oplus  &(e'Af\otimes hBg')\ \ \oplus\\
                                & (eAf'\otimes h'Bg) &\oplus & (e'Af'\otimes h'Bg').
\end{array}
\end{equation}
If the first summand is non zero, then $eAf\neq 0$ and $hBg\neq 0$ and by definition there are corresponding arrows in the $E$ and Peirce $F$-quivers, respectively from $f$ to $e$ and from $g$ to $h$.  We associate to this non zero summand the following $E$-vertical balanced cycle with revolution number $1$:
\begin{itemize}
\item[--] the first vertical down arrow from $e$ to $g$ corresponds to the  projective direct summand $Bg\otimes eA$ of $M$,
\item[--] the subsequent horizontal arrow at the ground floor from $g$ to $h$, due to  $hBg\neq 0$,
\item[--] next the vertical up arrow from $h$ to $f$ which corresponds to the projective bimodule $Af\otimes hB$,
\item[--] finally the horizontal arrow at the first floor from $f$ to $e$, due to $eAf\neq 0$.
\end{itemize}
The decomposition of $H_0\left(A, (N\otimes_B M)^{\otimes_A 2}\right)$ contains the following direct summand
\begin{equation}\label{T2NM}
\begin{array}{llll}
 (eAf'\otimes h'Bg' \otimes e'Af \otimes hBg).
\end{array}
\end{equation}
It corresponds to the $E$-vertical balanced cycle $\gamma$ with revolution number $2$, described by the following sequence of vertices (from right to left) and drawn below:
$$e,f',h',g',e',f,h,g,e$$

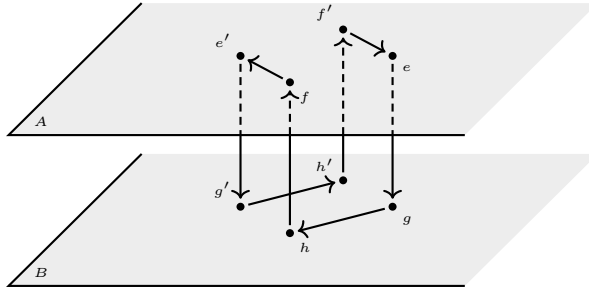
\begin{figure}[h]

\begin{center}
\begin{tikzpicture}[
        x={(-0.35cm,-0.35cm)},
        y={(1cm,0cm)},
        z={(0cm,1cm)},
        font=\tiny,
        ]
  \fill[black!7] (-4,0,0) -- (1,0,0) -- (1,6,0) -- (-4,6,0) -- cycle;
  \fill[black!7] (-4,0,2) -- (1,0,2) -- (1,6,2) -- (-4,6,2) -- cycle;
  \draw[thick] (-4,0,0) -- (1,0,0) -- (1,6,0);
  \draw[thick] (-4,0,2) -- (1,0,2) -- (1,6,2);
  \coordinate (g') at (-2,2,0);
  \coordinate (h') at (-3,3,0);
  \coordinate (h) at (-1,3,0);
  \coordinate (g) at (-2,4,0);
  \coordinate (e') at (-2,2,2);
  \coordinate (f') at (-3,3,2);
  \coordinate (f) at (-1,3,2);
  \coordinate (e) at (-2,4,2);
  \foreach \n in {h, g, e, f}
    {
    \fill (\n) circle(1.5pt);
    \node[below right] at (\n) {$\n$};
    }
  \foreach \n in {h', g', e', f'}
    {
    \fill (\n) circle(1.5pt);
    \node[above left] at (\n) {$\n$};
    }
  \begin{scope}[
        shorten <=3pt,
        shorten >=3pt,
        thick
        ]
  \foreach \a / \b in {g/h, f/e', g'/h', f'/e}
    \draw[->] (\a) -- (\b);
  \foreach \a / \b in {h/f, h'/f'}
    {
    \draw[shorten >=0pt] (\a) -- (\a |- 1,0,2);
    \draw[->, shorten <=0pt, densely dashed] (\a |- 1,0,2) -- (\b);
    }
  \foreach \a / \b in {e'/g', e/g}
    {
    \draw[shorten >=0pt, densely dashed] (\a) -- (\a |- 1,0,2);
    \draw[->, shorten <=0pt] (\a |- 1,0,2) -- (\b);
    }
  \node at (0.5,0.25,0) {$B$};
  \node at (0.5,0.25,2) {$A$};
  \end{scope}
\end{tikzpicture}
\end{center}
\caption{$(N,M)$\sf -quiver}\label{NM quiver}
\end{figure}

The direct summands of $H_0\left(A, (N\otimes_B M)^{\otimes_A 2}\right)$ are originated by the indecomposable direct summands of $M$ and $N$. The vertical arrows of the $E$-vertical balanced cycle  keep track of them. For instance  the vertical arrow from $e'$ to $g'$ corresponds to the projective direct summand $Bg'\otimes e'A$ of $M$.

Note  that the $E$-vertical balanced cycle drawn above is not the square of a vertical balanced cycle of revolution number $1$. On the other hand, $E$-vertical balanced cycles which are powers of shorter ones do exist.

Let $\gamma\in \mathsf{CV}^E_m$. We consider the non zero vector subspaces of $A$ and $B$ corresponding to the horizontal arrows of $\gamma$, which belong to the respective $E$ and Peirce $F$-quivers. Let  $V_\gamma$ be  their tensor product, obtained by following the order of the arrows of $\gamma$.

Conversely, as sketched above, a non zero  vector space direct summand of  $H_0\left(A, (N\otimes_B M)^{\otimes_A m}\right)$ determines an $E$-vertical balanced cycle of revolution number $m$. Then  $$H_0\left(A, (N\otimes_B M)^{\otimes_A m}\right)=\bigoplus_{\gamma\in \mathsf{CV}^E_m} V_\gamma.$$

We describe now the transported action of $C_m$ on $\mathsf{CV}^E_m$.  Let $\gamma$ be an $E$-vertical balanced cycle at a vertex $e$, of revolution number $m$. Let  $\gamma'$ be the $E$-balanced path obtained from $\gamma$ by  removing at its beginning the balanced oriented path $\alpha$ defined as follows: $\alpha$ is the first arrow of $\gamma$ followed by the next ones until reaching the source of the second vertical down arrow of $\gamma$. Note that $\alpha$ begins at $e$, it has revolution number $1$, and in general $\alpha$ is not  a cycle.  The target of $\gamma'$ is still $e$, and we have $t\cdot\gamma=\alpha\gamma'$.

We suppose now $H_0\left(A, (N\otimes_B M)^{\otimes_A m}\right)\neq 0$, that is $\mathsf{CV}^E_m\neq\emptyset$. Let $\gamma\in \mathsf{CV}^E_m$ and let $\underline{\gamma}$ be the $E$-vertical balanced cycle of smallest length such that $\gamma=(\underline{\gamma})^l$, in particular $\underline{\gamma}$ is not a power of a shorter $E$-vertical balanced cycle.

The stabilizer subgroup of $\gamma$ in $C_m$ is generated by $t^{\frac{m}{l}}$. That is  $t^{\frac{m}{l}}\cdot\gamma = \gamma$ and $\{t^i \gamma\}_{i=0,\dots, \frac{m}{l}-1}$ are distinct. Let $k[\mathsf{CV}^E_m]$ be the vector space with basis $\mathsf{CV}^E_m$. The trace element $\hat{\gamma} = \gamma +t\cdot\gamma+t^2\cdot\gamma+\cdots +t^{\frac{m}{l}-1}\cdot\gamma\in k[\mathsf{CV}^E_m]$   is a sum of different basis elements, hence $\hat{\gamma}\neq 0$. Moreover $t\cdot\hat{\gamma} = \hat{\gamma}$.  We will infer from $\underline{\gamma}$ a non zero  element of $H_0\left(A, (N\otimes_B M)^{\otimes_A m}\right)$.

 Let $\underline{u}$ be a non zero  element of $V_{\underline{\gamma}}$, and let $u=\underline{u}^{\otimes l}\in V_\gamma$.  This way $u\neq 0$ and  $t^{\frac{m}{l}}\cdot u=u$. Moreover $t^i\cdot u\in V_{t^i\gamma}$. Observe that the vector spaces $V_{t^i\cdot\gamma}$ are distinct for $i=0,\dots, \frac{m}{l}-1$ since the corresponding $E$-vertical balanced paths are  different. Consequently $\hat{v}=v+t\cdot v+\dots +t^{l-1}\cdot v\neq 0$. Moreover $t\cdot\hat{v}=\hat{v}$, then $H_0\left(A, (N\otimes_B M)^{\otimes_A m}\right)^{C_m}\neq 0$.
\qed
\end{proof}

\begin{theo}\label{NMNMNM=0}
Let $\Lambda=\left(
                                                                      \begin{array}{cc}
                                                                        A & N \\
                                                                        M & B\\
                                                                      \end{array}
                                                                    \right)$
be a null-square projective algebra where $A$ and $B$ are basic finite dimensional algebras over a perfect field $k$, and let $M$ and $N$ be  finitely generated projective bimodules, given as in (\ref{M}) and (\ref{N}).
If $HH_n(\Lambda)=0$ for $n$ large enough, then $H_0 \left(A, (N\otimes_BM)^{\otimes\!_{_A}n}\right)=0$ for all $n>0$ and $\left(N\otimes_B M\right)^{\otimes\!_{_A}n}=0$ for $n$ large enough.
\end{theo}
\begin{proof}
The hypothesis that Hochschild homology of $\Lambda$ vanishes in large enough degrees,  imply by Corollary \ref{invariantszzero} that $H_0\left(A, (N\otimes_B M)^{\otimes_A m}\right)^{C_m}=0$ for $m$ large enough. By the previous theorem \ref{invariantszero}, we infer $H_0\left(A, (N\otimes_B M)^{\otimes_A m}\right)=0$ for the same $m$'s, hence $\mathsf{CV}^E_m=\emptyset$ for $m$ large enough.

However if  $\mathsf{CV}^E_{n_0}\neq\emptyset$ for some $n_0$, then $\mathsf{CV}^E_{rn_0}\neq\emptyset$  for all $r>0$. Hence $\mathsf{CV}^E_n=\emptyset$ for all $n>0$. As a consequence $H_0\left(A, (N\otimes_B M)^{\otimes_A n}\right)=0$ for all $n>0$.

We assert  that in the same way as in the proof of Theorem \ref{invariantszero}, $\left(N\otimes_B M\right)^{\otimes_A m}$ is a direct sum of non zero  vector spaces which are in one-to-one correspondence with $\mathsf{P}^E_m$ (see Definition \ref{balanced}). For instance in the decomposition (\ref{NM}), the first summand $Af\otimes hBg \otimes eA$ corresponds to the $E$-balanced paths which contain the vertical down arrow from $e$ to $g$ and the vertical up arrow from $h$ to $f$. More precisely, there is a subsequent decomposition
 $$Af\otimes hBg \otimes eA = \bigoplus_{y,x\in E} yAf\otimes hBg \otimes eAx,$$
 and for each non zero  summand $yAf\otimes hBg \otimes eAx$, the $E$-balanced path is determined by the sequence of vertices $y,f,h,g,e,x.$

 In particular  $\left(N\otimes_B M\right)^{\otimes_A m}=0$ if and only if $\mathsf{P}_m^E=\emptyset$.

 We have  shown before that the $(N,M)$-quiver has no $E$-balanced vertical cycles. Since the $(N,M)$-quiver is finite, the $E$-balanced paths have a maximal length. Then $\mathsf{P}_n^E=\emptyset$ for $n$ large enough, and $\left(N\otimes_B M\right)^{\otimes\!_{_A}n}=0$ for the same set of $n$'s. \qed
\end{proof}

The long exact sequence of Theorem \ref{longexactsequence} provides then the following
\begin{coro}\label{iguales}
Let $\Lambda=\left(
                                                                      \begin{array}{cc}
                                                                        A & N \\
                                                                        M & B\\
                                                                      \end{array}
                                                                    \right)$
be a null-square projective algebra where $A$ and $B$ are basic finite dimensional algebras over a perfect field $k$, where $M$ and $N$ are finitely generated projective bimodules.
If $HH_n(\Lambda)=0$ for $n$ large enough, then for all $n$
$$HH_n(\Lambda)= HH_n(A)\oplus HH_n(B).$$

\end{coro}

Our next aim is to provide a tool for bounding above the global dimension of a null-square projective algebra.
For this purpose we first briefly recall the \emph{mapping cone} construction. Let $(C_\bullet,c) = \{C_n\stackrel{c_n}{\to} C_{n-1}\}_{n\in \mathbb{Z}}$ and $(D_\bullet,d)=\{D_n\stackrel{d_n}{\to} D_{n-1}\}_{n\in \mathbb{Z}}$ be complexes with differentials $c$ and $d$. Let $f:C_\bullet \to D_\bullet$ be a map of complexes. Let $C_\bullet[1]$ be the complex defined by $C_n[1]= C_{n-1}$. There exists a  complex $(\mathsf{co}(f)_\bullet,e)$ called the mapping cone of $f$, and a short exact sequence of complexes
$$0\to C_\bullet[1]\to \mathsf{co}(f)_\bullet \to D_\bullet \to 0$$
such that the connecting homomorphism in the  long exact sequence of cohomology is the morphism induced by $f$. In particular, $f$ induces isomorphisms (\emph{i.e.} $f$ is a \emph{quasi-isomorphism}) if and only if the mapping cone complex is acyclic. Actually $\mathsf{co}(f)_n=C_n\oplus D_{n-1}$ with differential  $e=\left(
                                                                                                                                                                 \begin{array}{rr}
                                                                                                                                                                   c & f \\
                                                                                                                                                                   0& -d \\
                                                                                                                                                                 \end{array}
                                                                                                                                                               \right)$; note that the change of sign for $d$ guarantees $e^2=0$, since $fc=df$.

We simplify the tensor product notation as follows: let $U$ be a  $C\!-\! B$-bimodule and let $V$ be a $B\!-\!A$-bimodule, we will write $UV$ instead of $U\otimes_B V$ and $VU$ instead of $V\otimes_A U$.

\begin{theo}\label{null-squaremodulefgd}
Let $\Lambda=\left(
                                                                      \begin{array}{cc}
                                                                        A & N \\
                                                                        M & B\\
                                                                      \end{array}
                                                                    \right)$
be  a null-square  projective algebra where $A$ and $B$ are k-algebras, and $M$ and $N$ are $B\!-\! A$ and $A\!-\!B$-projective bimodules respectively.  Let $X$ be a left $A$-module and $P_\bullet \to X$ be a projective resolution.

Associated to $P_\bullet \to X$, there is a $\Lambda$-projective resolution $Q_\bullet \to (X\rightleftharpoons 0)$ such that if $P_\bullet \to X$ is finite and if $\left(N\otimes_B M\right)^{\otimes\!_{_A}n}=0$ for $n$ large enough, then $Q_\bullet \to (X\rightleftharpoons 0)$ is finite.
\end{theo}
\begin{proof}
We define the modules of $Q_\bullet$  as follows:
\vskip5mm
{\footnotesize
\hskip-1cm
\def\arraystretch{2.6}
$\begin{array}{lllll}
Q_0= \left(P_0 \underset{0}   {\overset{1}{\rightleftharpoons}} MP_0\right)\\
Q_1= \left(P_1 \underset{0}   {\overset{1}{\rightleftharpoons}} MP_1\right) \oplus \left(NMP_0 \underset{1}   {\overset{0}{\rightleftharpoons}} MP_0\right)\\
Q_2= \left(P_2 \underset{0}   {\overset{1}{\rightleftharpoons}} MP_2\right) \oplus \left(NMP_1 \underset{1}   {\overset{0}{\rightleftharpoons}} MP_1\right)\oplus \left(NMP_0 \underset{0}   {\overset{1}{\rightleftharpoons}} M(NM)P_0\right)\\
Q_3 = \left(P_3 \underset{0}   {\overset{1}{\rightleftharpoons}} MP_3\right) \oplus \left(NMP_2 \underset{1}   {\overset{0}{\rightleftharpoons}} MP_2\right)\oplus  \left(NMP_1 \underset{0}   {\overset{1}{\rightleftharpoons}} M(NM)P_1\right)\oplus
\left((NM)^2P_0 \underset{1}   {\overset{0}{\rightleftharpoons}} M(NM)P_0\right)\\
Q_4 = \left(P_4 \underset{0}   {\overset{1}{\rightleftharpoons}} MP_4\right) \oplus \left(NMP_3 \underset{1}   {\overset{0}{\rightleftharpoons}} MP_3\right)\oplus \left(NMP_2 \underset{0}   {\overset{1}{\rightleftharpoons}} M(NM)P_2\right) \oplus \\ \left((NM)^2P_1 \underset{1}   {\overset{0}{\rightleftharpoons}} M(NM)P_1\right)\oplus
\left((NM)^2P_0 \underset{0}   {\overset{1}{\rightleftharpoons}} M(NM)^2P_0\right)\\

\vdots\\
Q_{2m}=\left(P_{2m} \underset{0}   {\overset{1}{\rightleftharpoons}} MP_{2m}\right) \oplus \left(NMP_{2m-1} \underset{1}   {\overset{0}{\rightleftharpoons}} MP_{2m-1}\right)\oplus\cdots \oplus
\left((NM)^{m}P_0 \underset{0}   {\overset{1}{\rightleftharpoons}} M(NM)^{m}P_0\right)\\
Q_{2m+1}=\left(P_{2m+1} \underset{0}   {\overset{1}{\rightleftharpoons}} MP_{2m+1}\right) \oplus \left(NMP_{2m} \underset{1}   {\overset{0}{\rightleftharpoons}} MP_{2m}\right)\oplus\cdots \oplus 
\left((NM)^{m+1}P_0 \underset{1}   {\overset{0}{\rightleftharpoons}} M(NM)^{m}P_0\right)\\
\vdots\\
\end{array}$}

 We observe that the $Q_i$ are projective $\Lambda$-modules. Indeed, first we note that the   free rank one bimodule $B\otimes A$ is  projective as a left (or right module), hence any projective bimodule (for instance $M$) is projective as a left (or right module).  Consequently for any left $A$-module $X$, the left $B$-module $M\otimes_A X$ is projective. Finally, Lemma \ref{projectiveone} shows that each direct summand of $Q_i$ is a projective $\Lambda$-module.

 The differentials are defined in the figure below:

\begin{landscape}
\begin{tikzcd}[column sep=0.15]
\vdots & & \vdots & & \vdots & & \vdots & & \vdots  &  & \vdots & & \vdots & & \vdots & & \vdots & & \vdots \\
(NM)^2P_0 \ar[ddrr, swap, "1"] \ar[rrrrrrrrrrrrrrrrrr, rightharpoonup, bend left=8] & \oplus & (NM)^2P_1 \ar[dd, "-1\mbox{\tiny $\otimes$}p_1"]  & \oplus & NMP_2  \ar[dd, "1\mbox{\tiny $\otimes$}p_2"] \ar[rrrrrrrrrr, rightharpoonup, bend left=10] \ar[ddrr,  swap, "1"] & \oplus & NMP_3 \ar[dd, "-1\mbox{\tiny $\otimes$}p_3"] &\oplus & P_4 \ar[dd, "{p_4}"] &  \rightharpoonup  & MP_4 \ar[dd, "1\mbox{\tiny $\otimes$}p_4"] & \oplus &  MP_3  \ar[llllll, , rightharpoonup, bend left=20] \ar[dd,  "-1\mbox{\tiny $\otimes$}p_3"] \ar[ddll, "1"] & \oplus &   M(NM)P_2 \ar[dd,  "1\mbox{\tiny $\otimes$}p_2"]  & \oplus & M(NM)P_1 \ar[dd,  "-1\mbox{\tiny $\otimes$}p_1"] \ar[ddll, "1"] \ar[llllllllllllll, rightharpoonup, bend left=9] & \oplus & M(NM)^2P_0 \\
& & & & & & & &  &   &   &   &   & & & & & & \\
& & (NM)^2P_0  & \oplus & NMP_1  \ar[dd,  "1\mbox{\tiny $\otimes$}p_1"] \ar[rrrrrrrrrr, rightharpoonup, bend left=15] \ar[ddrr, swap, "1"] & \oplus & NMP_2 \ar[dd,  "-1\mbox{\tiny $\otimes$}p_2"] &\oplus & P_3 \ar[dd,  "{p_3}"] &  \rightharpoonup  & MP_3 \ar[dd,  "1\mbox{\tiny $\otimes$}p_3"] & \oplus &  MP_2  \ar[llllll, rightharpoonup, bend left=20] \ar[dd,  "-1\mbox{\tiny $\otimes$}p_2"] \ar[ddll, "1"] & \oplus &   M(NM)P_1 \ar[dd,  "1\mbox{\tiny $\otimes$}p_1"]  & \oplus & M(NM)P_0  \ar[ddll, "1"] \ar[llllllllllllll, rightharpoonup, bend left=10] & & \\
& & & & & & & &  &   &   &   &   & & & & & & \\
& & & & NMP_0 \ar[rrrrrrrrrr, rightharpoonup, bend left=10] \ar[ddrr, swap, "1"] & \oplus & NMP_1 \ar[dd,  "-1\mbox{\tiny $\otimes$}p_1"] &\oplus & P_2 \ar[dd,  "{p_2}"] &  \rightharpoonup  & MP_2 \ar[dd, "1\mbox{\tiny $\otimes$}p_2"] & \oplus &  MP_1  \ar[llllll, rightharpoonup, bend left=20] \ar[dd, "-1\mbox{\tiny $\otimes$}p_1"] \ar[ddll, "1"] & \oplus &   M(NM)P_0   & & & &  \\
& & & & & & & &  &   &   &   &   & & & & & & \\
& & & & & & NMP_0 & \oplus &  P_1 \ar[dd, "{p_1}"] & \rightharpoonup  &  MP_1  \ar[dd, "1\mbox{\tiny $\otimes$}p_1"] & \oplus &  MP_0  \ar[llllll, rightharpoonup, bend left=20] \ar[ddll, "1"]  &  & & & & &  \\
& & & & & & & &  &   &   &   &   & & & & & & \\
& & & &  & & & & P_0 \ar[d, "{p_0}"] & \rightharpoonup  &  MP_0 \ar[d, "1\mbox{\tiny $\otimes$}p_0"] &  & &  & & & & & \\
& & &  & & & & & X  \ar[d] &  \rightharpoonup  & 0 \ar[d]   &  & &  & & & & &  \\
& & & & & & & &  0 &   &  0 &   &   & & & & & &
\end{tikzcd}
\end{landscape}

It is immediate to check that the differentials are morphisms of $\Lambda$-modules,  that is, the corresponding squares commute (see Definition \ref{categoryS} and Proposition \ref{modulesandcategoryS}).

The  column with $X$ in the bottom is the projective resolution $P_\bullet \to X$. We observe that the two columns on its right give the mapping cone of the identity of the complex $(MP_\bullet, -1{\tiny\otimes}p)$. Since the identity is an isomorphism, the mapping cone is exact. Similarly, the next two columns on the right provide the mapping cone of the identity for the complex $(M(NM)P_\bullet, -1{\tiny\otimes}p)$, and so forth.

The two columns on the left of  $P_\bullet \to X$ correspond to the mapping cone of the identity of the complex $(NMP_\bullet, 1{\scriptsize\otimes}p)$. The next two columns on the left are the mapping cone of the identity of $\left((NM)^2P_\bullet, 1{\tiny\otimes}p\right)$, and so forth.

Consequently  $Q_\bullet$  is a resolution of  $X\rightleftharpoons 0$ by projective $\Lambda$-modules.

Let $r$ be an integer such that $(NM)^i=0$ for $i>r$. Moreover let $l$ be an integer such that $P_j=0$ for $j>l$. For a given $m$, the module $Q_{m}$ is the direct sum of vector spaces of the form $(NM)^iP_j$ for $2i+j=m$ or $2i+j=m+1$, and of vector spaces of the form $M(NM)^iP_j$ for $2i+j=m$ or $2i+j=m+1$. Let $m>2r+l$. In case $2i+j=m$ or $m+1$,  either $i>r$ or $j>l$ since otherwise $2i+j\leq 2r+l<m$. Hence $Q_m=0$ for all $m>2r+l$.
\qed
\end{proof}

\begin{theo}\label{smooth}
Let $k$ be a perfect field and let $\Lambda=\left(
                                                                      \begin{array}{cc}
                                                                        A & N \\
                                                                        M & B\\
                                                                      \end{array}
                                                                    \right)$
be  a finite dimensional null-square projective algebra where $A$ and $B$ are smooth.

If $(NM)^n=0$ for large enough $n$, then $\Lambda$ is smooth.
\end{theo}
\begin{proof}
The complete list of simple $\Lambda$-modules is
$ \{S \underset{}   {\overset{}{\rightleftharpoons}} 0 \}\bigcup \{0 \underset{}   {\overset{}{\rightleftharpoons}} T\}$ where $S$ and $T$ are simple modules over $A$ and $B$ respectively, see Proposition \ref{simples}. The previous theorem
shows that $S \underset{}   {\overset{}{\rightleftharpoons}} 0$ is of finite projective dimension. The analogous theorem holds for $\Lambda$-modules of the form $0 \underset{}   {\overset{}{\rightleftharpoons}} Y$ where $Y$ is a $B$-module.  Then the simple modules $0 \underset{}   {\overset{}{\rightleftharpoons}}T$ are also of finite projective dimension. \qed
\end{proof}

\begin{theo}
Let $k$ be a perfect field. Any finite dimensional null-square projective  $k$-algebra built on the class $\H$ of basic $k$-algebras verifying Han's conjecture also belongs to $\H$.
 \end{theo}
\begin{proof}
Let $\Lambda=\left(
                                                                      \begin{array}{cc}
                                                                        A & N \\
                                                                        M & B\\
                                                                      \end{array}
                                                                    \right)$,
                                                                    where $A$ and $B$ are finite dimensional basic $k$-algebras which belong to $\H$, and $M$ and $N$ are projective bimodules. Suppose $HH_n(\Lambda)=0$ for $n$ large enough. Then by Corollary \ref{iguales}, $HH_n(A)$ and $HH_n(B)$ vanish for the same set of $n$'s, hence $A$ and $B$ are smooth. Moreover, by Theorem \ref{NMNMNM=0} we have $(NM)^n=0$ for $n$ large enough. The previous corollary shows that $\Lambda$ is smooth. \qed
\end{proof}

\begin{rema} As much as in Remark \ref{agree}, we observe that according to Corollary \ref{iguales} this result agrees with the property proved by B. Keller in \cite[2.5]{KELLER}, namely the Hochschild homology of a finite dimensional smooth algebra over a perfect field is concentrated in degree zero.
\end{rema}

\section{\sf Gabriel quiver and relations of a null-square projective algebra}

Let $A$ be a finite dimensional algebra such that $A/\rad A$ is a product of copies of $k$, in other words $A$ is basic - equivalently, $A$ is Morita reduced - and sober - that is, the algebra of $A$-endomorphisms of each simple $A$-module is just $k$. Let $E$ be a complete system of primitive and orthogonal idempotents. The set of vertices of the \emph{Gabriel quiver} $Q_A$ is $E$; the number of arrows from $x$ to $y$ is the dimension of the vector space $y (\rad A / \rad^2 A)x$. It is well known that $Q_A$ is canonical, in the sense that $Q_A$ does not depend on the choice of $E$.

Let $Q$ be a quiver with finite set of vertices $Q_0$ and set of arrows $Q_1$. The vector space $kQ_0$ is endowed with a semisimple algebra structure where $Q_0$ is a complete set of primitive orthogonal idempotents. Note that $kQ_0$ is basic and sober. The vector space $kQ_1$ is a $kQ_0$-bimodule in the natural way. The \emph{path algebra} $kQ$ is by definition the tensor algebra $T_{kQ_0}(kQ_1)$, it has a canonical basis given by the oriented paths of $Q$. The universal property of $kQ$ is as follows: any algebra map $\varphi : kQ\to X$ is  determined by an algebra map $\varphi_{0}: kQ_0\to X$ - that is, a set map from $Q_0$ to a system of $X$ -, and $\varphi_1 : kQ_1\to X$, a $kQ_0$-bimodule map - the structure of $X$ as $kQ_0$-bimodule being  inferred from $\varphi_0$.

A finite dimensional algebra $A$ as above can be \emph{presented}, namely there exists a - non canonical - algebra surjection $kQ\to A$ such that its kernel $I$ is an admissible two-sided ideal,  that is, there exist a positive integer $m$ such that $F^m\subset I\subset F^2$, where $F$ is the two sided ideal generated by $(Q_A)_1$. Moreover, the ideal $I$ decomposes as $ \oplus_{x,y\in E}\  yIx$ since $(Q_A)_0$ is complete. The system of generators $R$ of $I$ considered in a presentation is \emph{adapted},  that is, $R$ is graded with respect to this  decomposition, its elements are called relations. Note that any system of generators $R'$ gives rise to a graded one, namely $R= \bigsqcup_{x,y \in E} yR'x$, where for a set of paths $Z$, we denote by $yZx$ the paths of $Z$ starting at $x$ and ending at $y$.

Let $\Lambda=\left(
               \begin{array}{cc}
                 A & N \\
                 M & B \\
               \end{array}
             \right)$
be a finite dimensional null-square projective algebra, where $A$ and $B$ are basic and sober, with respective presentations $(Q_A,R_A)$ and $(Q_B,R_B)$, and where the projective bimodules $M$ and $N$ are given as in (\ref{M}) and (\ref{N}).

\begin{lemm}
The Gabriel quiver of $\Lambda$ is the disjoint union of $Q_A$, $Q_B$ and new arrows as follows:
\begin{itemize}
\item ${}_g m_e $  arrows from $e\in E$ to $g\in F$, which we call \emph{down arrows},
\item ${}_f n_h $  arrows from $h\in F$ to $f\in E$, which we call \emph{up arrows}.
\end{itemize}
\end{lemm}
\begin{proof}
The description of the Jacobson radical of $\Lambda$  as given in the proof of Lemma \ref{simples} provides immediately the result. \qed
\end{proof}

Let $T_{A\times B}(M\oplus N)$ be the tensor algebra of the $A\times B$-bimodule $M\oplus N$, where, as already mentioned, the given actions are extended by zero in order to consider $M$ and $N$ as $A\times B$-bimodules. For instance we infer $M\otimes_{A\times B} M= N\otimes_{A\times B} N=0$.

The next two results are easy to prove, using both that $M$ and $N$ are projective bimodules, and the universal properties of the algebras involved.

\begin{lemm}\label{iso}
There is an algebra isomorphism $\varphi : T_{A\times B}(M\oplus N)\to kQ_{\Lambda}/\langle R_A, R_B\rangle$.
\end{lemm}

\begin{lemm}
Let $\psi: T_{A\times B}(M\oplus N) \to \Lambda$ be the algebra map given by the inclusions of $A\times B$ and $M\oplus N$.
$$\Ker \psi = \langle (M\oplus N)^{\otimes_{(A\times B)}2} \rangle= \langle N\otimes_B M + M\otimes_A N\rangle.$$
\end{lemm}

\vskip2mm
 The set of all oriented paths of $Q_A$ generates the vector space $kQ_A/\langle R_A\rangle $, hence we can choose a subset $P_A$ which is a basis of $kQ_A/\langle R_A\rangle $. Let also  $P_B$ be a basis of $kQ_B/\langle R_B\rangle $, where $P_B$ is  a subset of the oriented paths of $Q_B$.

Let  $u$ be a down arrow from $e$ to $g$, and let $v$ be an up arrow from $h$ to $f$  in $Q_\Lambda$. We define the sets $v \curlyvee u$ and $u\curlyvee v$ of oriented paths of $Q_\Lambda$  as follows:
$$v\curlyvee u = v(hP_Bg)u \mbox{\ \ and \ \ }   u\curlyvee v=   u(eP_Af)v.$$

Let $R$ be the disjoint union of $R_A$, $R_B$, and $v\curlyvee u$ and $u\curlyvee v$ for all pairs $(u,v)$, where $u$ is a down arrow and $v$ is an up arrow.

\begin{theo}
Let $\Lambda=\left(
               \begin{array}{cc}
                 A & N \\
                 M & B \\
               \end{array}
             \right)$
be a finite dimensional null-square projective algebra, where $A$ and $B$ are basic and sober algebras with presentations $(Q_A, R_A)$ and $(Q_B, R_B)$ respectively,  and where the projective bimodules $M$ and $N$ are given as in (\ref{M}) and (\ref{N}).
The algebra $\Lambda$ is presented by $(Q_\Lambda, R)$.
\end{theo}
\begin{proof}
The key point of the proof is the following. Consider the image of $\Ker \psi$ by $\varphi$ (see Lemma \ref{iso}) in $kQ_\Lambda/\langle R_A, R_B\rangle$. Let $Bg\otimes eA$ be a direct summand of $M$, and $Af \otimes hB$ be a direct summand of $N$.  They provide the direct summand  $Bg\otimes eAf \otimes hB$ of $M\otimes_A N\subset \left(T_{A\times B}(M\oplus N)\right)_2 \subset \Ker \psi$. In order to consider its image by $\varphi$, let $u$ and $v$ be the arrows in $Q_\Lambda$ associated respectively to $Bg\otimes eA$ and $Af \otimes hB$.  The image of $Bg\otimes eAf \otimes hB$  in $kQ_\Lambda/\langle R_A, R_B\rangle$ is generated by $u\curlyvee v$. \qed
\end{proof}

\begin{exam}
Let $Q_A$ be a \emph{crown} quiver with three arrows $a_0$, $a_1$ and $a_2$, that is these arrows start respectively at $e_0$, $e_1$ and $e_2$, and end respectively at $e_1$, $e_2$ and $e_0$. Let $R_A=\{a_2a_0\}$. It is easy to establish that $A=kQ_A/\langle R_A \rangle$ is smooth. Let $(Q_B, R_B)$ be a presentation of a basic and sober algebra $B$, and let $g$ and $h$ be vertices of $B$.

Let $M=Bh\otimes e_1A$ and $N=Ae_2\otimes gB$, and $u$ from $e_1$ to $h$ and $v$ from $g$ to $e_2$ the corresponding arrows. Note that the $(N,M)$-quiver has no oriented cycles. Let $\Lambda$ be the corresponding null-square projective algebra, next we describe its Gabriel quiver and a set of relations:

\begin{itemize}
  \item $Q_\Lambda = Q_A \cup Q_B \cup \{u,v\}$.
  \item $R=\{a_2a_0, R_B\}\cup\{v\gamma u\}_{\gamma \in gP_Bh}$, where $P_B$ is basis of oriented paths of $B$.
\end{itemize}

Moreover it follows from the previous results that if $B$ is smooth, then $\Lambda$ is smooth.
\end{exam}

\vskip5mm

\noindent\textbf{Acknowledgements:} the first and third authors thank Marcelo Lanzilotta and  Universidad de la Rep\'ublica (Uruguay) for excellent conditions during part of the preparation of this work. We also thank Cristian Chaparro and Mariano Su\'arez  \'Alvarez for help with tikz figures.

\footnotesize
\noindent C.C.:
\\IMAG - Institut Montpelli\'{e}rain Alexander Grothendieck\\
Univ Montpellier, CNRS, Montpellier, France.\\
{\tt Claude.Cibils@umontpellier.fr}

\medskip

\noindent M.J.R.:
\\Instituto de Matem\'{a}tica (INMABB, UNS/CONICET)\\
Universidad Nacional del Sur\\
Av. Alem 1253\\
8000, Bah\'\i a Blanca, Argentina.\\{\tt mredondo@criba.edu.ar}

\medskip
\noindent A.S.:
\\IMAS-CONICET y Departamento de Matem\'atica,
 Facultad de Ciencias Ezactas y Naturales,\\
 Universidad de Buenos Aires,
\\Ciudad Universitaria, Pabell\'on 1\\
1428, Buenos Aires, Argentina. \\{\tt asolotar@dm.uba.ar}

\end{document}